\documentclass[reqno]{amsart}
\usepackage{fullpage}

\usepackage{algpseudocode}

\usepackage{todonotes}
\usepackage{thm-restate}

\usepackage[
backend=biber,
style=numeric,
maxbibnames=99
]{biblatex}

\usepackage[]{amsfonts}
\usepackage{color}
\usepackage[]{amsmath}
\usepackage{amssymb}
\usepackage[]{amsthm}
\usepackage{graphicx}
\usepackage{enumerate}
\usepackage{hyperref}
\usepackage{thmtools}
\usepackage{tikz}
\usepackage{subcaption}
\usepackage{thm-restate}
\usepackage{mathtools}

\title{Combinatorial interpretation of the coefficients of the order polynomial of fence posets}

\author[Yakob Kahane]{Yakob Kahane}
\address[Y.~Kahane]{LACIM, Université du Québec à Montréal, Canada}
\email{kahane.yakob@courrier.uqam.ca}

\newtheorem{theorem}{Theorem}[section]
\newtheorem*{theorem*}{Theorem}
\newtheorem{lemma}[theorem]{Lemma}
\newtheorem*{lemma*}{Lemma}
\newtheorem{corollary}[theorem]{Corollary}
\newtheorem*{corollary*}{Corollary}

\newtheorem*{conjecture*}{Conjecture}
\newtheorem{proposition}[theorem]{Proposition}
\newtheorem*{proposition*}{Proposition}

\newtheorem{remark}[theorem]{Remark}

\newcommand{\new}[1]{\textit{\textbf{\color{blue}{#1}}}}

\theoremstyle{definition}
\newtheorem{definition}[theorem]{Definition}

\numberwithin{equation}{section}

\begin{document}

\begin{abstract}
Given a fence poset $P$, we define a new statistic on permutations, denoted by $bl_P$, that provides a combinatorial interpretation of the coefficients of the order polynomial of $P$, answering a question of Ferroni, Morales, and Panova (2025). Using the fact that the base polytope of a lattice path matroid can be decomposed into order polytopes of fence posets, we also obtain a combinatorial interpretation of the coefficients of the Ehrhart polynomial of the base polytope of Schubert matroids, answering a question of Stanley about the hypersimplex (1999). As an application of this statistic, we establish the first nontrivial lower bound for the linear coefficient of the Ehrhart polynomial of the order polytope of fence posets. Finally, we conjecture generalizations of this statistic to skew shape posets and circular fence posets.
\end{abstract}

\maketitle

\section{Introduction}
Let $P$ be a poset of size $n$. We are interested in the coefficients of the order polynomial which count the number of increasing weakly increasing functions from $P$ to $\{1, \dots, t \}$.  
\[
\Omega(P;t)=\sum_{k=1}^{n} c_k(P)\, t^k .
\]
This polynomial encodes substantial information about $P$. In particular, the leading coefficient of $n!\cdot \Omega(P;t)$ is the number of linear extensions of $P$, denoted $e(P)$ (see \cite[Ch.~3]{EC2}), a central and notoriously difficult to compute invariant \cite{brigthwellwrinckler}.  

From a geometric viewpoint, following Stanley \cite{stanley1986two}, we define the \emph{order polytope}
\[
\mathcal{O}(P):=\{ f\in [0,1]^P : f(u)\le f(v)\ \text{whenever }u\preceq v \text{ in } P\}.
\]
The Ehrhart polynomial of $\mathcal{O}(P)$ is equal to $\Omega(P;t+1)$. It is well known that the coefficients $a_k := n!\, c_k(P)$ are integers, though they are not always nonnegative. Determining when these coefficients are positive remains an open problem \cite{liu2019positivity} \cite{liu2019stanley}. A recent result of Ferroni, Morales, and Panova \cite{ferroni2025skew} shows that to prove nonnegativity of all coefficients, it suffices to establish positivity of the linear coefficient of the order polynomial of all connected subposets of~$P$. In \cite[Theorem 3.6]{kahane2025yoshida}, we extended their result and expressed the coefficient $c_k$ as the sum of products of linear coefficient of order polynomial of subposets reformulating a result of Shareshian, Wright and Zhao \cite{shareshian2003newapproachorderpolynomials}. 
For fence posets, skew shape posets and circular fence posets (see Figure \ref{fig:example_posets} for some examples), all the coefficients $a_k$ are non-negative integers \cite{ferroni2025skew}, and $\sum_{k=0}^{n} a_k = n! \Omega(1) = n!$. Ferroni, Morales and Panova \cite{ferroni2025skew} asked whether we can interpret the coefficients $a_k$ as a statistic on the set of permutation $S_n$.

\begin{figure}[ht]
    \centering
    \includegraphics[width=0.5\linewidth]{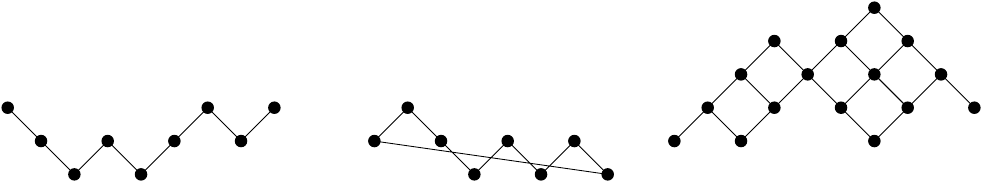}
    \caption{The fence poset of size $9$ with descents at position $\{2,3,5,8\}$, the circular fence poset of size $8$ with descents at position $\{3,4,6,8\}$ and the skew shape poset $P_{\lambda /\mu}$ with $\lambda /\mu = 544221/11$.}
    \label{fig:example_posets}
\end{figure}

Inspired by decompositions into connected subposet, we define new candidate statistics which could give an interpretation of the coefficients $a_k = n!c_k$. This approach gives positive results for fence posets. Given a fence poset $P$ and a permutation, we introduced a statistic called "block statistic" $bl_P(\sigma)$ in order to obtain 

\begin{restatable*}[]{theorem}{blockformula}
	\label{thm:Omega_to_bl}
    Let $P$ be a fence poset of size $n$, then 
    \[
    n! \, \Omega(P,t) = \sum_{\sigma \in S_n}t^{bl_P(\sigma)}. 
    \]
\end{restatable*}
In the case of the chain poset $P = [n]$, we recover the Foata correspondence. 

Using the same ideas we also introduced analogs for this statistics, $\widehat{bl}_P(\sigma)$ for circular fence posets  and $\widetilde{bl}_P(\sigma)$ for skew posets. We conjectured the two following statements\footnote{Shortly after the first version of this paper was posted which included these conjectures, Huang in \cite{huang2026bernsteintransfersgreedyrecords} settled the first one for circular fence posets via different methods. This also allowed him to give a new proof of Theorem~\ref{thm:Omega_to_bl}.}  

\begin{restatable*}[]{conjecture}{CircularFenceConj}[{Now a theorem by Huang}]
    \label{conj:circular_poset_stat}
	Let $P$ be a circular fence poset of size $n$, then 
    \[
    n! \; \Omega(P; t) = \sum_{\sigma \in S_n(P)} t^{\widehat{bl}_P(\sigma)}. 
    \]
\end{restatable*}

and 

\begin{restatable*}[]{conjecture}{skewshapeConj}
    \label{conj:skew shape stat}
	Let $P$ be a skew shape poset of size $n$. Then
	\[
	n! \, \Omega(P;t) = \sum_{\sigma \in S_n(P)} t^{\widetilde{bl}_P(\sigma)}.
	\]
\end{restatable*}

Recently, Ferroni, Morales, and Panova \cite{ferroni2026ehrhartpositivitylatticepath} proved positivity of the coefficients of the Ehrhart polynomial  of the base polytope of lattice path matroids $M[L,U]$. These are important matroids whose bases are given by lattice paths between two lattice paths $L$ and $U$. Their proof relies on a decomposition of the Ehrhart polynomial into a sum of order polynomials of certain fence posets. We reformulate their result by expressing the Ehrhart polynomial of the base polytope of a lattice path matroid denoted by  $\mathbb{Ehr}(M[L,U];t)$ as a sum of polynomials indexed by the bases $\mathcal{B}(M[L,U])$ of the matroid, using the \new{shade} of a lattice path (see Section~\ref{sec:LPM} for precise definition and Figure \ref{fig:shade_example} for an illustration).

\begin{figure}
\centering
\includegraphics[width=0.2\linewidth]{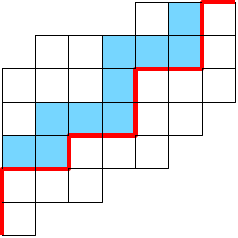}
\caption{The shade $sh(x)$ of the lattice path $NNEENEENNEENNE$.}
\label{fig:shade_example}
\end{figure}

\begin{restatable*}[]{corollary}{ehrhartLPM}
    \label{coro:ehrLPM}
    Let $M[L,U]$ be a lattice path matroid associated to the the skew shape $\lambda /\mu $, then 
     \[    \mathbf{Ehr}{(M[L,U]; t)} =  \sum_{x \in \mathcal{B}(M[L,U])} \sum_{\sigma \in S_{|\mathsf{sh}(x)|}} \frac{t^{bl_{\mathsf{sh}(x)}(\sigma)}}{|\mathsf{sh}(x)|!},
     \]
where $|\mathsf{sh}(x)|$ denotes the number of cells of $\mathsf{sh}(x)$. 
\end{restatable*}

Schubert matroids form an important subclass of lattice path matroids. They arise naturally in algebraic geometry \cite{BergetEurSpinkTseng}, combinatorics \cite{BonindeMierNoy,FerroniSchubert}, and representation theory \cite{FinkMeszarosStDizier}. Specializing Corollary \ref{coro:ehrLPM} to Schubert matroids allows us to introduce a new statistic, $bl^{*}_P(\sigma)$, which provides a combinatorial interpretation of the order polynomial of their base polytopes.

\begin{restatable*}[]{theorem}{schubertLPM}
    \label{thm:schubertLPM}
    Let $U$ be a lattice path and let $M[U]$ be the associated Schubert matroid. Then
\[
\mathbf{Ehr}(M[U];t) = \frac{1}{n!}\sum_{x \in \mathcal{B}(M[U])} \sum_{\sigma \in S_{n}}  t^{bl^{*}_{P_x}(\sigma)}.
\] 
\end{restatable*}

Uniform matroids $U_{k,n}$ are Schubert matroids corresponding to the lattice path $U$ consisting of $k$ north steps followed by $n-k$ East steps. Their base polytopes are the hypersimplices $\Delta_{k,n}$. In this case, Theorem \ref{thm:schubertLPM} specializes to the following result.

\begin{restatable*}[]{proposition}{hypersimplexLPM}
    \label{prop:hypersimplexLPM}
     The Ehrhart polynomial of the hyper-simplex $\Delta_{n, k}$ can be expressed as 
    \[
    \mathbf{Ehr}(\Delta_{k,n};t) = \frac{1}{n!}\sum_{P} \sum_{\sigma \in S_{n}} t^{\mathsf{Tr}Bl_P(\sigma)},
    \]
where the first sum ranges over all the fence posets $P$ of size $n+1$ with exactly $k$ descents. 
\end{restatable*}

\section*{Outline} 
The paper is organized as follows. Section $2$ gives all the background and notations useful to understand the rest of the paper. Section $3$ introduce the block statistic $bl_P(\sigma)$ which is the main result of this article. Section $4$ gives two natural variations of the block statistic in the case of circular fence posets and skew shape posets. This gives two conjectures about the coefficients of the order polynomial of those posets. Section $5$ gives a probabilistic application of the the block statistic which leads to a lower bound of the linear coefficient of the Ehrhart polynomial of the order polytope of the corresponding fence poset. In Section $6$, we first express the Ehrhart polynomial of the base polytope of lattice path matroid using a recent result from Ferroni, Morales and Panova \cite{ferroni2026ehrhartpositivitylatticepath}, and we use this formula to introduce a new truncated block statistic $bl_P(\sigma)$ that gives a combinatorial interpretation of the coefficients of the Ehrhart polynomial of the base polytope of lattice path matroid.

\section{Background and notation}

We assume that the reader is familiar with basic definitions of posets and matroids, we recall all the concepts useful for the comprehension of this article. 

\subsection{Skew shape posets}

Let $\lambda=(\lambda_1,\ldots,\lambda_r)$ and $\mu=(\mu_1,\ldots,\mu_r)$ be
partitions such that $\mu_i\leq\lambda_i$ for every $i$. The \new{skew shape}
$\lambda/\mu$ is the set of cells
\[
\lambda/\mu
=
\{(i,j)\in\mathbb{Z}_{>0}^2:
1\leq i\leq r,\ \mu_i<j\leq\lambda_i\}.
\]

The associated \new{skew shape poset} is the poset
\[
P_{\lambda/\mu}=(\lambda/\mu,\leq_P),
\]
where the order is the restriction of the coordinate-wise order:
\[
(i,j)\leq_P(i',j')
\quad\Longleftrightarrow\quad
i\leq i'\ \text{and}\ j\leq j'.
\]

Equivalently, the cover relations are
\[
(i,j)\lessdot_P(i+1,j)
\qquad\text{and}\qquad
(i,j)\lessdot_P(i,j+1),
\]
whenever both cells belong to $\lambda/\mu$.
\begin{figure}[ht]
	\centering
	\includegraphics[width=0.4\linewidth]{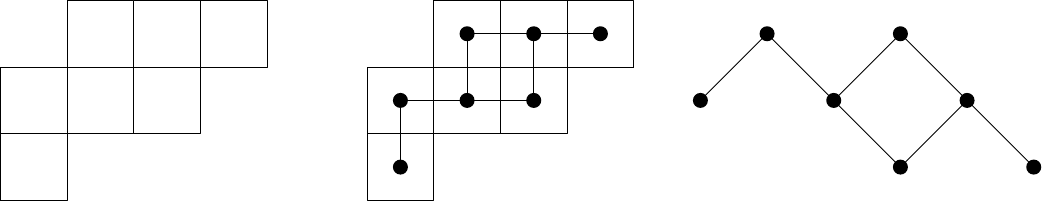}
	\caption{The skew shape $431/1$ and its associated skew shape poset.}
	\label{skew_shape_poset}
\end{figure}
\subsection{Order Polynomial and Ehrhart polynomial}

\begin{definition}
    Let $P = (X,\preceq)$ be a partially ordered set (poset) with $|X| = n$.  
    The \new{order polynomial} of $P$ is the unique polynomial \new{$\Omega(P;t)$} $\in \mathbb{Q}[t]$ such that
    \begin{equation}
        \Omega(P;t) := \#\{ f : X \to [1,\dots,t] \;\text{ such that }\; f(x) \le f(y) \text{ whenever } x \preceq y \}.
    \end{equation}
\end{definition}

It is a classical result (see for example \cite{stanley2011enumerative}) that this counting function is a polynomial because it can be expressed as a sum over chains of ideals. 
\begin{equation}
\label{eqn:poly_binomial}
    \Omega(P;t) = \sum_{k=1}^{n} \sum_{\emptyset= I_0 \subsetneq I_1 \dots \subsetneq I_k = P } \binom{t}{k}.
\end{equation}

Equation \eqref{eqn:poly_binomial} implies that $n!\cdot \Omega(P;t)$ has integer coefficients. Thus, we may write:
\[
    \sum_{k=0}^n a_k(P)t^k := n!\cdot \Omega(P;t),
    \qquad
    \sum_{k=0}^n c_k(P)t^k := \Omega(P;t).
\]

Using Equation \eqref{eqn:poly_binomial}, we can also determine  the leading coefficient $a_n(P)$.

\begin{proposition}
	\label{prop:linear_extension}
	Let $P$ be a poset of size $n$, let $e(P)$ be the number of linear extension of $P$, then 
	\[
	a_n(P) = e(P).
	\] 
\end{proposition}

An \new{ideal} (resp. \new{filter}) of $P$ is a set $S \subset P$, such that if $x \in P, y \in S$ with $x \leq y$ (resp. $x \geq y)$ then $x \in S$. Notice that given a weakly increasing function $f$ from $P$ to $[1, \dots ,t+1]$, $f^{-1}([1, \dots, t])$ is ideal of $P$ and $f^{-1}([2,t+1])$ is a filter of $P$. Hence 
\[ 
\Omega(P;t+1) = \sum_{S \textnormal{ is an ideal of }P} \Omega(S;t) = \sum_{S \textnormal{ is a filter of }P} \Omega(S;t).
\]
\begin{definition}
    The \new{order polytope} of $P$, introduced by Stanley \cite{stanley1986two}, is defined as
    \begin{equation}
        \mathcal{O}(P) := \{ f \in [0,1]^P \;\text{ such that }\; f(u) \le f(v) \text{ whenever } u \preceq v \}.
    \end{equation}
\end{definition}

The polytope $\mathcal{O}(P)$ is integral. Stanley showed in \cite[Thm.~4.1]{stanley1986two} that the \new{Ehrhart polynomial}  of $\mathcal{O}(P)$, which counts the lattice points in the dilation $t\cdot \mathcal{O}(P)$, is given by the order polynomial.

\begin{proposition}[{\cite[]{stanley1986two}}]
    For any poset $P$, $\mathcal{E}(P)$, the Ehrhart polynomial of $\mathcal{O}(P)$ satisfies
    \[
        \mathcal{E}(P;t) := \#(t\cdot \mathcal{O}(P) \cap \mathbb{Z}^n) = \Omega_P(t+1).
    \]
\end{proposition}

\vspace{5mm}

Ferroni, Morales, and Panova \cite[Prop.~3.3]{ferroni2025skew} showed by induction that the non-negativity of the coefficients of the order polynomial of a poset $P$ follows from the positivity of the linear coefficient of the order polynomial of every connected subposet of $P$.
Moreover, they proved the positivity of those coefficients for skew shape posets.

\begin{theorem}{\cite[Prop. 4.1]{ferroni2025skew}}
\label{thm:formula_c1_FMP}
    Let $\lambda/\mu$ a skew shape of size $n$ and length with order polynomial $\Omega(P_{\lambda/\mu}; t) = \sum_{k=1}^{n}
c_k(P)t^k$. Then if $\lambda/\mu$ is not connected, i.e, $\lambda_i = \mu_i$ or $\lambda_i \leq \mu_{i-1}$ for some i we have $c_1(P_{\lambda/\mu}) =0$. If $\lambda/\mu$ is connected then \[c_1(P_{\lambda/\mu}) = \frac{(\lambda_1 - \mu_{\ell}-1)!(\ell-1)!}{(\lambda_1 - \mu_{\ell}+ \ell -1)!}.\]
\end{theorem}

\subsection{Fence posets and cyclic fence posets}

A \new{fence poset} $P$ is a poset whose Hasse diagram is a path.  
It can be described by a collection of positions $2 \leq \alpha_1 < \alpha_2 < \dots < \alpha_k \leq n$ on the ground set $\{z_1, \dots, z_n\}$ such that  
\begin{equation} 
	\label{eq: ineqs fences}
	z_1 \prec z_2 \prec \dots \prec z_{\alpha_1 - 1} \succ z_{\alpha_1} \prec \dots \prec z_{\alpha_2 - 1} \succ z_{\alpha_2} \prec \dots
\end{equation}

For a given fence poset $P$, we introduce the following notations which will be convenient for the proof of the main theorem. 

Let : 
\begin{enumerate}
	\item \textcolor{blue}{\textbf{desc}$(P)$} be the set of elements $x$ of $P$ which are lower than the element at its left. Those elements are named \textbf{descent elements}.
	\item \textcolor{blue}{\textbf{asc}$(P)$} be the set of elements $x$ of $P$ which are greater than the element at its left or are the leftmost element. Those elements are named \textbf{ascent elements}.
	\item \textcolor{blue}{$P_{/i}$} be the fence poset constructed by removing the element of $i$ and the edge at its left and gluing the remaining. 
	\item let \textcolor{blue}{$f(i)$} be the first position after $i$ in desc, we use the convention $f(i) = |P|$ if there is no such element. 
\end{enumerate}

For example, in Figure \ref{fig:example_f} the ascent elements are in blue ({\LARGE $\textcolor{blue}{\circ}$}) while the descent elements are in green ({\LARGE $\textcolor{green}{\bullet}$}) and we have $f(0) = f(1) = 2, f(3) = f(4)=5, f(6) = 7$ and $f(10) = f(11) = 12$. \\
Removing the vertex at position $5$ gives Figure \ref{fig:example_P/i}. 

\begin{figure}[ht]
	\centering
	\begin{minipage}{0.43\textwidth}
		\centering
		\includegraphics[width=\linewidth]{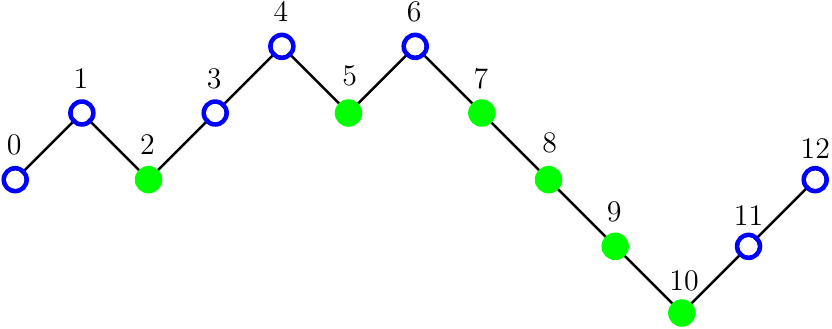}
		\caption{A fence poset $P$. In blue the ascent set, in green the descent set. }
		\label{fig:example_f}
	\end{minipage}
	\hfill
	\begin{minipage}{0.41\textwidth}
		\centering
		\includegraphics[width=\linewidth]{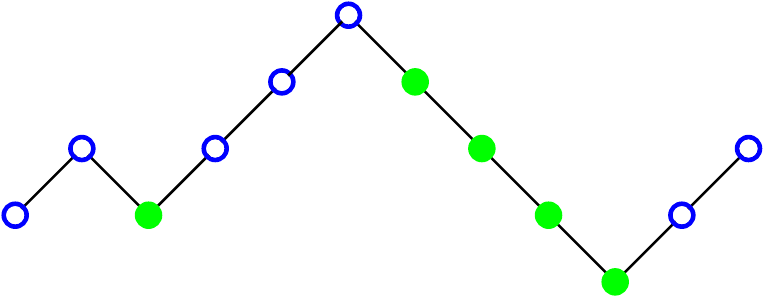}
		\caption{The fence poset $P_{/5}$.  }
		\label{fig:example_P/i}
	\end{minipage}
\end{figure}

Let $P$ be a fence poset of $n$ elements. The position of the elements of $P$ goes from left to right starting at $0$. Let $P_{[i,j]}$ be the fence subposet consisting of element with position $i \leq \ell \leq j$. Note that $P_{[0,n-1]} = P$, $P_{[j,i]} = \emptyset$ if $j  > i $ and $P_{[i,i]}$ is the singleton poset. 

If $P$ is a fence poset, the linear coefficient of the order polynomial has been computed by Ferroni, Morales, Panova.  

\begin{lemma}[{\cite[Prop.~3.3]{ferroni2025skew}}]
\label{lem:c1_fence}
    Let $P$ be a fence, let $r = |\textnormal{asc}(P)|$ and let $s = |\textnormal{asc}(P)|$. then 
    \[ 
    c_1(P) = \frac{(r-1)!s!}{(r+s)!}.\]
\end{lemma}

\subsection{Standard decompostions into subposets}
In this section, we introduce standard decompositions of posets into connected (i.e. the Hasse diagram is a connected graph) subposets. We then use these decompositions to express the coefficients of the order polynomial as a weighted sum over these decompositions.

\begin{definition}[decomposition into connected subposets] \label{def:2}
	Let $P$ be a poset, a \new{decomposition $\mu$ into connected subposets} is a partition of the ground set of $P$ into $k$ connected subposets $P_1,..., P_k$.
	Let $P_i, P_j$ two subposets of $P$ from the decomposition $\mu$, we say that $P_i \widetilde{\leq_{\mu}} P_j$ if and only if there exist $a_i \in P_i$ and $a_j \in P_j$ such that $a_i \leq_P a_j$. 
\end{definition}

\begin{figure}[ht]
    \centering
    \includegraphics[width=0.7\linewidth]{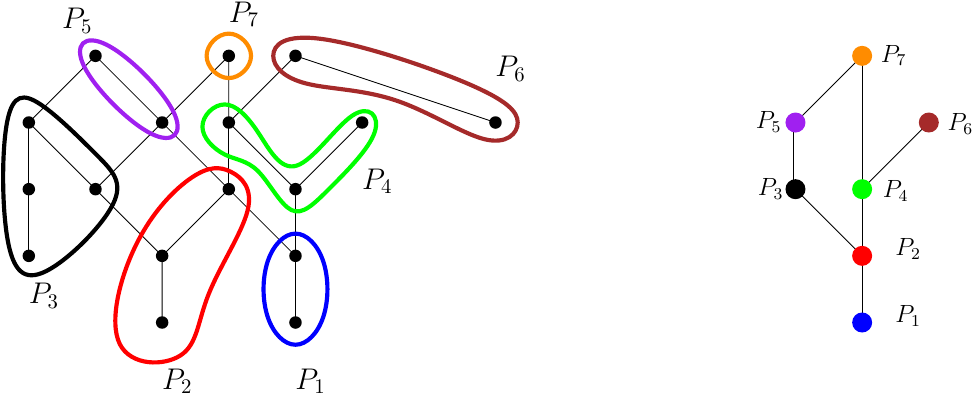}
    \caption{A standard decomposition $\mu$ and the induced poset.}
    \label{fig:standard_decompo}
\end{figure}

From the definition of subposets, we see that $P_i \widetilde{\leq}_{\mu} P_j$ if and only if $i=j$. However the relation $\widetilde{\leq}_{\mu}$ is not a partial order relation yet (for example in the left of Figure \ref{fig:standard_decompo}, $P_1 \widetilde{\leq}_{\mu} P_2$ and $P_2 \widetilde{\leq}_{\mu} P_3$, but there is no relation between $P_1$ and $P_3$). In order to create such a relation we need to consider a smaller class of decomposition into connected subposets. 

\begin{definition}[Standard Decomposition]\label{def:3}
	A decomposition $\mu = \{P_1,...., P_k\}$ into connected subposets is said to be \new{standard} if there is no $2 \leq r \leq k$ and $r$ distinct $i_1, i_2,..., i_r$ such  that  $P_{i_1} \widetilde{\leq}_{\mu} P_{i_2} .... \widetilde{\leq}_{\mu} P_{i_r} \widetilde{\leq}_{\mu} P_{i_1}$. 
\end{definition}

\begin{proposition}\label{lemma:2}~\\
	Let $\leq_{\mu}$ be the transitive closure of $\widetilde{\leq}_{\mu}$ then : 
	\[\leq_{\mu} \text{ is a a partial order relation } \text{ if and only if } \mu \text{ is a standard decomposition.}\]
\end{proposition}
We define \new{$\text{SDecompo}_{k}(P)$} to be the set of standard decompositions of $P$ into $k$ connected subposets. Figure \ref{fig:standard_decompo} gives an example of a standard decomposition into $7$ subposets.

\vspace{5mm}

Coefficients of the order polynomial $\Omega(P;t)$ can be express as a weighted sum over the subposets of $P$ where the weight function is the product of the linear coefficient of the block posets which form the decomposition. In this case, we define \new{$e(\mu)$} to be the number of linear extension on ${(\mu, \leq_{\mu})}$. It was first discovered by Shareshian, Wright, Zhao (\cite{shareshian2003newapproachorderpolynomials}), using chains of ideals. In this article, we will restate the theorem using standard decompositions into connected subposets introduced by the author in \cite{kahane2025yoshida}.

\begin{theorem}[{\cite[Proposition 2.10]{shareshian2003newapproachorderpolynomials}}]
Let $P$ be a poset of size $n$. The coefficients of the order polynomial satisfy :  
\label{thm:formulmain}
    $$c_k(P) = \frac{1}{k!} \sum_{ \mu \in \text{SDecompo}_k(P)} e(\mu)\prod_{P_i \in \mu} c_1(P_i).$$
\end{theorem}
Given a poset $P$ of size $n$, a \new{labeling} is a bijective map from $P$ to $\{ 1, \dots, n \}$. Since this set is in bijection with $S_n$, the set of permutation of $\{1, \dots, n \}$, we shall identify it with $S_n$ whenever convenient.

\section{Statistic for fence Posets}
Given a fence poset $P$ of size $n$, recall that $\sum_{k=1}^{n}a_k = n! \cdot \Omega(P;t) = n!$, which is the number of labelings of $P$. Thanks to Lemma \ref{lem:c1_fence} and Theorem \ref{thm:formulmain}, those coefficients are nonnegative. The goal of this section is to give a statistic of the labeling of the vertices which gives an interpretation of the coefficients $a_k$. 

The interpretation of $[t^n] \, n!\; \Omega(P;t) = e(P)$ as counting linear extensions of $P$ (see Proposition \ref{prop:linear_extension}), together with the formula for $[t^1] \, \Omega(P;t)$ from Lemma \ref{lem:c1_fence} and the decomposition into connected subposets in Theorem~\ref{thm:formulmain}, inspired the following statistic describing the coefficients $a_k$. This statistic uses the notions of block partitions that will be defined in this section. Figure~\ref{fig:fence_blocks} gives an example of blocks and the statistic.

\begin{definition}[Fence blocks]
	Let $P$ be a fence poset and let $\sigma$ be a labeling of the elements of $P$.  
	A \new{block} is a subposet such that :
	\begin{itemize}
		\item The leftmost element $x$ of the block is called the \new{root},
		\item For each other element $y$ of the block, if $y \in \textnormal{asc}(P)$, then $\sigma(y) < \sigma(x)$,
		\item For each other element $y$ of the block, if $y \in \textnormal{desc}(P)$ then $\sigma(y) > \sigma(x)$.
	\end{itemize}
\end{definition}

A \new{block partition} of a poset and labeling $(P, \sigma)$ is a partition $B = \{B_1, \dots, B_k\}$ into blocks. In a similar way as for decomposition into subposets, given two blocks $B_i, B_j$ we say that $B_i < B_j$ if and only if there exists $y \in B_i, z \in B_j$ such that $y <_P z$. We denote the root of $B_i$ by  $x_i$. 

\begin{definition}
\label{def:order_relation_blocks}
    Given a fence poset with a labeling $(P, \sigma)$ and a block partition $B = \{B_1, \dots, B_k\}$. Define the order relations on the blocks by $B_i \textcolor{blue}{\prec} B_j$ if and only if $\sigma(x_i ) < \sigma(x_j)$. 
\end{definition}

We say that a block partition $B = \{B_1, \dots, B_k\}$ is \new{valid} if $i \neq j$, $B_i < B_j$ implies $B_i \prec B_j$.
Figure \ref{fig:fence_blocks} gives an example of a block and a valid decomposition of a labeled fence post $(P, \sigma)$. In fact, this valid decomposition is unique.  
\begin{lemma}
\label{lem:valide_block_fence}
    Let $P$ be a fence poset of size $n$ and a labeling $\sigma$, There exists an unique valid block partition on $(P, \sigma)$. 
\end{lemma}
\begin{proof}
    We proceed on induction on $n$, the size of $P$.
    If $n = 1$, the statement holds, since there is the unique trivial block partition which is valid. \\
    Let $n \geq 2$ such that the statement holds for all posets of size $n-1$. Let $P$ be a fence poset of size $n$ and let $y$ be the rightmost element of $P$. By reversing the poset if necessary and replacing $\sigma$ by $n+1 - \sigma$ we may assume that $y \in \textnormal{asc}(P)$. Let $P' = P \setminus \{y \}$. First, we determine what necessary condition should a valid block partition fulfill; assume that there exists a valid block partition $B = \{B_1, \dots, B_k\}$ of $P$, then the restriction of $B$ to $(P', \sigma_{|P'})$ is also a valid block partition. By the induction hypothesis, let $B' = \{B'_1, \dots, B'_{k'}\}$ be the unique block partition of $(P', \sigma_{|P'})$ and let $x_{k'}$ be the root of the rightmost block of $B'$.
    If $\sigma(y) > \sigma(x_k')$, since $y \in \textnormal{asc}(P)$, $y$ cannot be in the block of root $B_k'$. Therefore $y$ is the root of the block only containing $y$. And indeed $B = \{B_1, \dots, B'_{k'}, \{y\}\}$ is a valid block partition of $(P, \sigma)$ and therefore, the only block partition of $(P, \sigma)$. 
    If $\sigma(y) < \sigma(x_k')$, since $y \in \textnormal{asc}(P)$, $y$ cannot be the root of its block because in this case we would have $B'_{k'} < \{y\}$ but $B'_{k'} \succ \{y\}$. Therefore, $y \in B'_{k'}$. And indeed $B = \{B_1, \dots, B'_{k'} \cup \{y\}\}$ is a valid block partition of $(P, \sigma)$ and therefore, the only block partition of $(P, \sigma)$. 
\end{proof}

\begin{figure}[ht]
	\centering
	\includegraphics[width=0.9\linewidth]{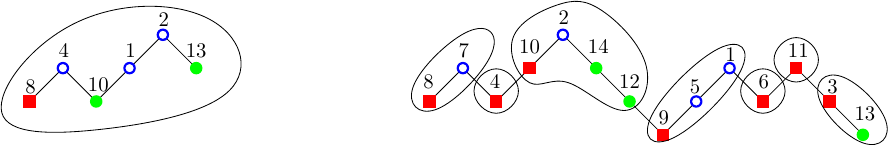}
	\caption{A block and the unique valid block partition $B(P,\sigma)$ of $(P,\sigma)$. Here, $bl_P(\sigma) = 7$. Roots are represented as red squares, ascents as blue circles and descents as green disks. }
	\label{fig:fence_blocks}
\end{figure}

Notice that the proof of the last lemma also gives an algorithm to construct \textcolor{blue}{$B(P, \sigma)$}, the unique valid block partition of a labeled fence poset $(P, \sigma)$. 
\begin{proposition}
	Given a fence poset $P$ with a labeling $\sigma$, the unique block partition of $(P, \sigma)$ is obtained by scanning from left to right and greedily constructing the largest possible leftmost block at each step.  
\end{proposition}

We can now define the polynomial $A(P;t)$, which will turn out to equal $\Omega(P;t)$ (see Theorem \ref{thm:Omega_to_bl}).

\begin{definition}
	\label{def:bl}
	Let $bl_P(\sigma)$ denote the number of blocks in the unique block decomposition of $(P,\sigma)$. We define $
	A(P;t) := \frac{1}{n!}\sum_{\sigma \in S_n(P)} t^{bl_P(\sigma)}$.
\end{definition}

Given a fence poset $P$ of size $n$, we derive the following recurrence for $A(P;t)$ by summing over all possible positions of the vertex labeled $n$.
\begin{lemma}
\label{lem:rec_a}
	Let $P$ be a fence poset of size $n$. The following equality holds
	\[n\cdot A(P_{[0,n-1]};t) = \sum_{i \in \text{desc}(P)}A(P_{[0,n-1]/i};t) + \sum _{i \in \text{asc}(P)} A(P_{[0,i-1]};t)\cdot t \cdot A(P_{[f(i),n-1]};t).\]
\end{lemma}

\begin{proof}
	Consider the position of the label $n$ in $P$, note that either $n$ is in $\mathrm{desc}(P)$—in which case it can be erased without changing the number of blocks—or $n$ is in $\mathrm{asc}(P)$, in which case it creates a new block ending at the first encountered element of $\mathrm{desc}(P)$. 
\end{proof}

\begin{figure}[ht]
	\centering
	\includegraphics[width=0.9\linewidth]{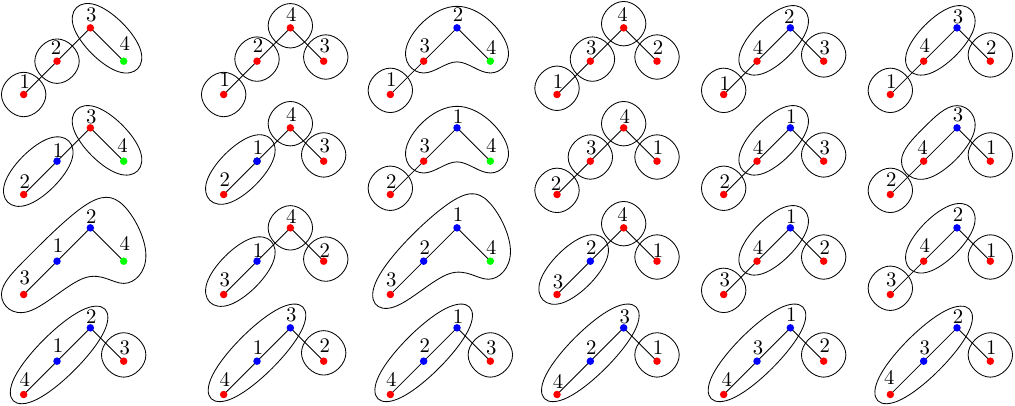}
	\caption{The $4!$ possible labelings for a fence poset of size $4$. Counting the number of blocks for each labeling gives $A(P;t) = \frac{1}{4!}(3t^4 + 10t^3 + 9t^2 + 2t)=\Omega(P;t)$.}
	\label{fig:all_labelings}
\end{figure}

\vspace{5mm}
Similarly, the order polynomial $\Omega(P;t)$ also satisfies a recurrence formula that directly comes from Theorem \ref{thm:formulmain}.

\begin{theorem}
\label{thm:rec_omega}
Let $P$ be a fence poset of size $n$, then, 
    \[\frac{d\,\Omega(P_{[0,n-1]};t)}{dt} = \sum_{\substack{0 \leq k < \ell \leq n \\ k \in \text{asc}(P), \;\ell \in \text{desc}(P) \cup \{ n\}}} \Omega(P_{[0,k-1]};t)\cdot c_1(P_{[k, \ell-1]};t) \cdot \Omega(P_{[\ell,n-1]};t).\]
\end{theorem}

\begin{proof}
	The proof follows by differentiating the decomposition formula for the order polynomial.
{\small
	\begin{align*}
		\frac{d\,\Omega(P_{[0,n-1]};t)}{dt} &=\frac{d}{dt}\Big(\sum_{k=0 }^{n} \frac{t^k}{k!}\!
		\sum_{\substack{\mu\in \mathsf{SDecompo}(P)\\ \mu=\{P_1,\dots,P_k\}}}e(\mu)\prod_{i=1}^{k}c_1(P_i)\Big) \\
		&= \Big(\sum_{k=0}^{n} \frac{t^{k-1}}{(k-1)!}\!
		\sum_{\substack{\mu\in \mathsf{Decompo}(P)\\ \mu=\{P_1,\dots,P_k\}}}e(\mu)\prod_{i=1}^{k}c_1(P_i)\Big).  \\
        \noalign{\noindent\textnormal{By setting $k'=k$ and $P_k=J$, we obtain}}
        \frac{d\,\Omega(P_{[0,n-1]};t)}{dt}&= \sum_{\substack{J \in \textnormal{coonected filter}(P) \\ J \neq \emptyset \quad }} c_1(J) \Big(\sum_{k'=0}^{n-1} \frac{t^{k'}}{k'!}\!
		\sum_{\substack{\mu\in \mathsf{Decompo}(P\setminus \{J \})\\ \mu=\{P_1,\dots,P_k'\}}}e(\mu)\prod_{i=1}^{k'}c_1(P_i)\Big). \\
        \noalign{\noindent\textnormal{Since there is no decomposition into $0$ subposets,}}
       \frac{d\,\Omega(P_{[0,n-1]};t)}{dt} &= \sum_{\substack{J \in \textnormal{connected filter}(P) \\ J \neq \emptyset \quad }} c_1(J) \cdot \Omega(P \setminus J ;t). \\
	\end{align*}
    let $k$ be the leftmost element of $J$ and $\ell$ be the rightmost element of $J$, then $k \in \textnormal{asc}(P)$ and $\ell \in \textnormal{desc}(P) \cup \{n \}$. Moreover, $P \setminus J$ is the product of $P_{[0, k-1]}$ and $P_{[\ell, n-1]}$. Therefore, 

\[\frac{d\,\Omega(P_{[0,n-1]};t)}{dt} = \sum_{\substack{0 \leq k < \ell \leq n \\ k \in \text{asc}(P), \;\ell \in \text{desc}(P) \cup \{ n\}}} \Omega(P_{[0,k-1]};t)\cdot c_1(P_{[k, \ell-1]};t) \cdot \Omega(P_{[\ell,n-1]};t),\]
as desired.

}
\end{proof}

By combining these two recurrence formulas with careful algebraic manipulations, we obtain the main theorem of this article. 

\blockformula

\begin{proof}
	Section \ref{section:proof_fence} is devoted to the proof of this result. 
\end{proof}

\begin{figure}
    \centering
    \includegraphics[width=0.3\linewidth]{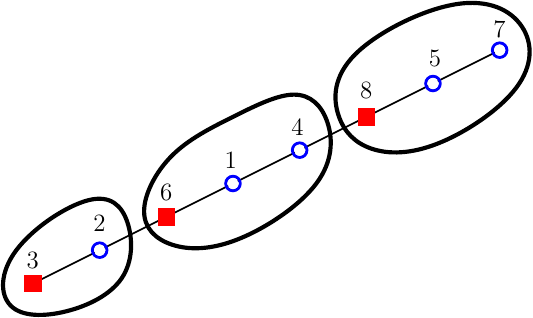}
    \caption{An example for the ascent chain case. The permutation $32614857$ gives the block partition $|32|614|857$ which match with the classical foata correspondance.}
    \label{fig:foata}
\end{figure}

\begin{remark}
    If the poset $P$ of size $n$ is just the chain $\mathbf{n}$ (see Figure \ref{fig:foata}), by reading the labels from left to right, a new block is created every time a label is greater than all the labels at its left. Therefore we recover the cycle notation used in the Foata correspondence, and therefore $n!\cdot \Omega([n];t) = \sum_{\sigma \in S_n} t^{|\textnormal{cycle}(\sigma)|}$.
\end{remark}

\section{Variations of the statistic}
\label{sec:variation}
In this section, we present two conjectures for variations of the statistics, the first one for all skew shape posets which are a generalization of fence posets and the second one for crown posets. 

\subsection{Circular fence poset}

\begin{definition}
    A circular fence poset $P$ of size $n$ is a poset whose Hasse diagram is a cycle. It can be described by a collection of position $2 \leq \alpha_1 < \dots < \alpha_k \leq n$ on the ground set $\{z_1, \dots, z_n\}$ such that 
    \begin{itemize}
        \item $\{\alpha_1, \dots, \alpha_{k-1} \} \neq \emptyset$,
        \item For all $1 \leq i < n$, and $i \notin \{\alpha_1, \dots, \alpha_{k} \}$, $z_i < z_{i+1}$, 
        \item for all $1 \leq i < n$, and $i \in \{\alpha_1, \dots, \alpha_{k \}}$, $z_i > z_{i+1}$,
        \item $z_{n} < z_1$. 
    \end{itemize}
\end{definition}

\begin{figure}[ht]
    \centering
    \includegraphics[width=0.3\linewidth]{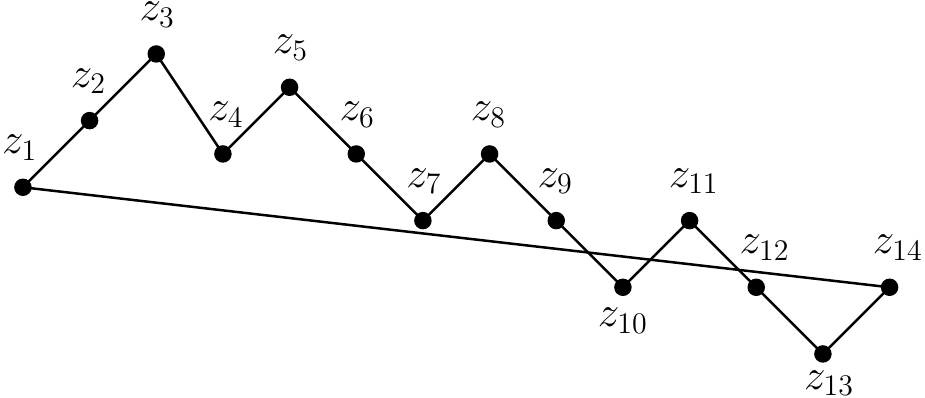}
    \caption{A circular fence poset of size $14$.}
    \label{circular_fence}
\end{figure}

We extend the definition of $\textnormal{desc}(P)$ and $\textnormal{asc}(P)$ on cycle cover naturally and saying that $z_1 \in \textnormal{asc}(P)$ (since it is greater that the ``element at its left" which is $z_n$). 

Connected subposets of a circular fence poset $P$ are either fence posets or $P$ itself. Given a labeling $\sigma$, we therefore extend the definition of block fences by saying that $P$ is a block if there exists an element $ x \in P$, such that for all $y$ in $P$,  $y \in \textnormal{desc(P)}$ if and only if $\sigma(y) > \sigma(x)$. In this case $x$ is the root of the block $P$. 

As for the skew shape case, given a partition into blocks $B = \{B_1, \dots, B_k\}$, we say that $B_i \prec B_j$ if and only if $\sigma(x_i) < \sigma(x_j)$ where $x_i$ and $x_j$ are respectively the roots of $B_i$ and $B_j$. We say that $B$ is a valid block partition if and only if $B_i < B_j$ implies that $B_i \prec B_j$. 

The next lemma will allow us to define a statistic for circular fence posets which is an analog of the block statistic of fence posets.  

\begin{lemma}
\label{lem:existence_valid_block}
    Let $P$ be a circular fence poset and $\sigma$ a labeling of its vertices. Then there exists a unique valid block partition on $(P, \sigma)$. 
\end{lemma}

\begin{proof}
    We proceed by induction on the size of $\textnormal{desc}(P)$.Let $P$ be a cyclic cover poset of size $n$ such that the lemma holds for all cyclic cover posets with smaller size of the descent set. Let $\sigma$ be a labeling of $P$ and $x \in P$ such that $\sigma(x) = n$. 
    \begin{itemize}
        \item If $x\in \textnormal{asc}(P)$. Then $x$ must be the root of its block. Let us now, consider $P'$ the fence poset derived from $P$, by removing the cover relation between $x$ and $y$, the element at its left. Using Lemma \ref{lem:valide_block_fence}, there exists a unique valid block decomposition $B = \{B_1, \dots, B_k\}$ of $(P', \sigma)$. Since $\sigma(x) = n$, $\sigma(x)$ is greater or equal than the label of the root containing $y$, therefore, $B$ is also the unique block decomposition of $P$.
        \item If $x\in \textnormal{desc}(P)$. Let $y$ be the element at the left of $x$ and $z$ the element at the right of $x$. The element $x$ cannot be the root of its block therefore $x$ is in the block containing $y$. Consider $P'$ the cyclic poset obtained by removing $x$, and adding the relation $y <_{P'} z$ if $x <_{P} z$ and $y >_{P'} z$ if $x >_{P} z$. By the induction hypothesis, let $B' = \{B_1', \dots, B_k'\}$ be the unique block decomposition of $(P', \sigma_{|P'})$. Let $i$ such that $y \in B'_i$. Let $B_j = \begin{cases}
            B_i \cup \{x \} \quad \textnormal{ if } j=i\\
            B_j \quad \quad \quad \quad \textnormal{otherwise}. 
        \end{cases}$, then $B = (B_1, \dots, B_k)$ is the unique block decomposition of $(P, \sigma)$. 
    \end{itemize}
    The base cases are the cases where $\textnormal{desc}(P) = \emptyset$ in which case $x$ is in  $\textnormal{asc}(P)$, hence the statement holds. 
\end{proof}

Given a circular fence poset $P$ and a labeling $\sigma$, we denote by $\widehat{bl}_P(\sigma)$ the size of the unique valid block partition of $(P, \sigma)$. 

\CircularFenceConj

This conjecture has been recently proven by Huang in \cite{huang2026bernsteintransfersgreedyrecords} using Bernstein transfer methods. 
\subsection{Skew shape posets}

Consider any skew shape $\lambda / \mu$, represented as a skew Young diagram.  
The poset $P_{\lambda / \mu}$ is the poset whose elements are the boxes $X(\lambda / \mu) = \{\, (i, j) \mid j \in [\mu_i + 1, \lambda_i] \,\},$ with covering relations given by 
$(i, j) \succeq (i + 1, j)
\quad \text{and} \quad
(i, j) \succeq (i, j + 1).$

Given a vertex associated to a cell $(i,j)$, we call respectively \new{left child}, \new{right child}, \new{left parent}, \new{right parent}, the vertices associated at respective position $(i+1,j)$, $(i,j+1)$, $(i,j-1)$, $(i-1,j)$.  Given a skew shape poset  $P$, a \new{leftmost element} of $P$ is a element of $P$ which doesn't have a left child nor a left parent. 

We gather in the next proposition some basic observations of skew shape that will be useful for defining the statistic. 
\begin{proposition}
\label{prop:skew_shape_propiete}
	Let $P$ be a connected skew shape poset.
	\begin{enumerate}
		\item $P$ has a unique left-most element 
		\item If an element $x \in P$ has a right parent $x_{+}$ and a right child $x_{-}$. Then there exists an element $y \in P$, such that $y$ is right child of $x_{+}$ and the right parent of $x_{-}$.  
	\end{enumerate}
\end{proposition}

For skew shape, inspired by Theorem \ref{thm:formulmain}, we will define the notion of blocks from skew shapes. Since down-sets and up-sets of skew shape posets are also skew shape posets, for every standard decomposition $\mu = (P_1, P_2, \dots, P_k)$, the $P_i$ are also skew shape posets.

\begin{definition}[Blocks in skew shapes]
	Let $P$ be a skew shape poset, and let $\sigma$ be a labeling of this poset.  
	A \emph{block} $B$ is a connected subposet of $P$ satisfying the following conditions (see Fig.~\ref{fig:skew_shape_ex_bloc} for an example):
	\begin{itemize}
        \item $B$ is a connected skew shape poset, its leftmost element is denoted by $x$ and is called the \new{root}, 
		\item for each element $y \in B$, if $y$ does not have a left parent, then $\sigma(y) < \sigma(x)$,
		\item for each element $y \in B$, if $y$ does not have a left child, then $\sigma(y) > \sigma(x)$. 
	\end{itemize}
\end{definition}

\begin{figure}[ht]
    \centering
    \begin{minipage}{0.39\textwidth}
        \centering
        \includegraphics[width=\linewidth]{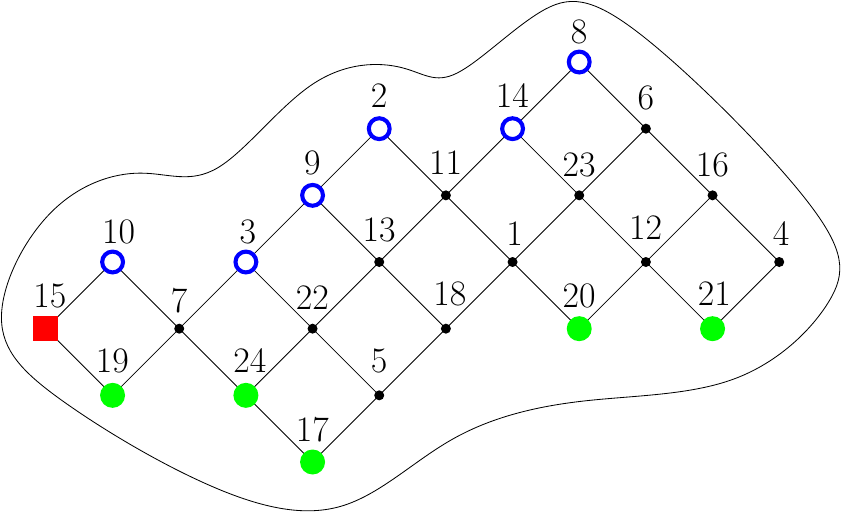}
        \caption{A block whose labeling satisfies the inequalities.}
        \label{fig:skew_shape_ex_bloc}
    \end{minipage}
    \hfill
    \begin{minipage}{0.6\textwidth}
        \centering
        \includegraphics[width=\linewidth]{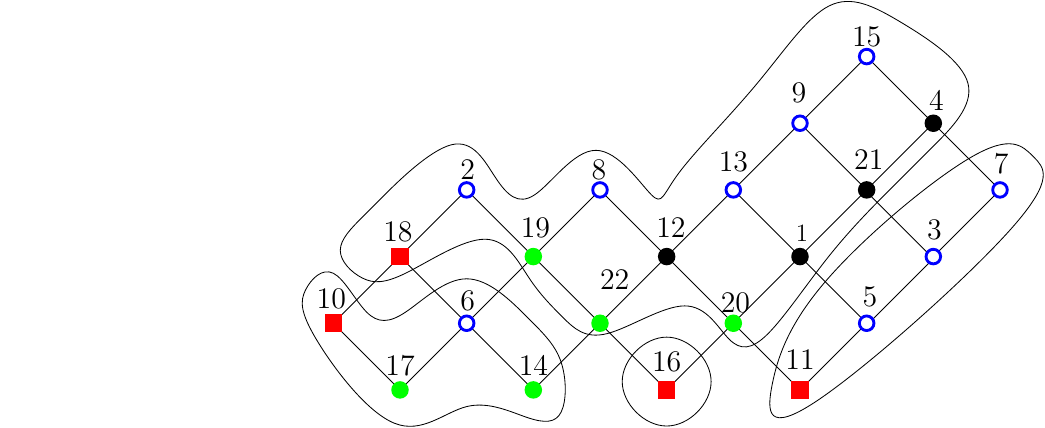}
        \caption{A valid partition into blocks. Notice that $10<18$, $16<18$, and $11<18$; thus, the partition is valid.}
        \label{fig:partition_skew_shape}
    \end{minipage}
\end{figure}

Given a block partition $B = (B_1, \dots, B_k)$, we denote by $x_i$ the root of $B_i$. We define the order relation $\prec$ and valid block partitions as for the fence poset case (see Definition \ref{def:order_relation_blocks}). Figure \ref{fig:partition_skew_shape} illustrates a valid block partition. In fact, this construction is unique as the next result shows.

\begin{lemma}
	Let $P$ a connected skew shape poset. And let $\sigma$ be a labeling of its vertices. Then, there exists a unique partition of the vertices which is a valid block partition. 
\end{lemma}

\begin{proof}
	We proceed by induction on $s$, the number of rows of the skew shape associated to $P$.\\
    For the base case, if $s = 1$, then, $P$ is a line, therefore it is a fence poset, the statement holds by Lemma \ref{lem:valide_block_fence}.\\ 
    For the induction case, assume that the statement holds for all posets with $s$ rows. By induction, let $P$ be a skew shape poset with $s+1$ rows and $\sigma$ a labeling of $P$. Take $Y = \{y_1, \dots, y_r \}$ the set of elements of $P$ whose associated cell in the skew shape have minimal first coordinate and sorted increasingly regarding the $j$ coordinate. Let $P' = P \setminus Y$. If $B = \{B_1, \dots, B_k\}$ is a valid block partition of $P$, then induced block partition on $P'$ is also a valid block partition. We will determine successively for each $y_j$ in which block they belong. Given an integer $i$ such that $1 \leq i \leq r$. let $x_+$ the root of the block containing $y_{i-1}$ (with the convention $\sigma(x_+) = + \infty$ if $i=1$). Let $x_-$ be the root of the block containing $z$, the left child of $y_i$ (with the convention $\sigma(x_-) = - \infty$ if $y_i$ has no left child). Notice that since the left child of $y_{i-1}$ is the left parent of $z$, $\sigma(x_-) \leq  \sigma(x_+)$. We have the following cases: 
    \begin{itemize}
        \item If $\sigma(x_-) = \sigma(x_+)$. It means that $y_{i-1}$ and $z$ are in the same block. Therefore, since using Proposition \ref{prop:skew_shape_propiete}, $y_i$ must also be in this block, and this create no problems. 
        \item If $\sigma(x_-) < \sigma(y_i) <  \sigma(x_+)$, then the only valid possibility is that $y_i$ is the root of its block. 
        \item If $\sigma(x_-)  <  \sigma(x_+) < \sigma(y_i)$, then the only valid possibility is that $y_i$ is block containing $x_+$. 
        \item If $\sigma(x_-)  <  \sigma(x_+) < \sigma(y_i)$, then the only valid possibility is that $y_i$ is block containing $x_+$. 
    \end{itemize}
    In all the cases, there is exactly one way to determine what is the root of the block containing $y_i$. Therefore, the block decomposition is uniquely determined. Figure \ref{fig:rec_skew_shape} illustrates the construction. 
\end{proof}
\begin{figure}[ht]
    \centering

    \begin{subfigure}{0.43\textwidth}
        \centering
        \includegraphics[width=\textwidth]{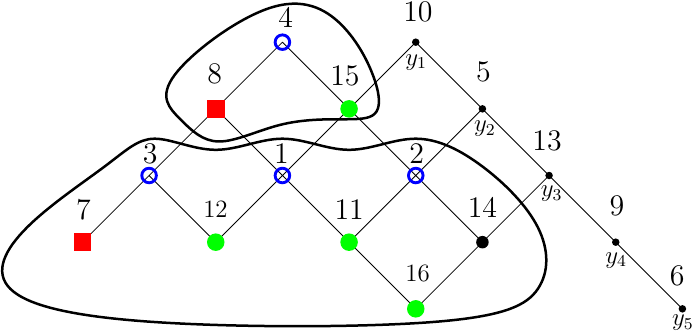}
        \caption{The block partition \\$B' = (\{ 1, 2, 3, 7, 11, 12, 14, 16\}, \{4, 8, 15\})$.}
    \end{subfigure}
    \hfill
    \begin{subfigure}{0.43\textwidth}
        \centering
        \includegraphics[width=\textwidth]{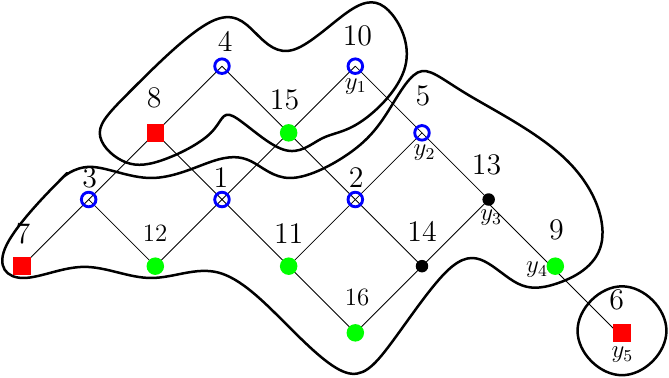}
        \caption{The new block partition $B$ of $\{1, \dots, 16\}$.}
    \end{subfigure}

    \caption{Illustration of the way of constructing  $B$ from $B'$. For example, for $i=2$, $\sigma(x_-) =7$ and since $y_1$ is in the block whose root has label $8$ $\sigma(x_+) = 8$, therefore, $y_2$ is in the same block as $x_-$.}
    \label{fig:rec_skew_shape}
\end{figure}

Given a skew shape poset $P$ and a labeling of its elements, we denote by \new{$\widetilde{bl}_P(\sigma)$} the number of blocks in the unique valid block decomposition of $(P,\sigma)$. Note that when $P$ is a fence poset, $\widetilde{bl}_P(\sigma)=bl_P(\sigma)$

\skewshapeConj

The conjecture have been checked for all skew shape posets of size smaller than $10$. 
\begin{remark}
	The leading term corresponds to the case where each block has size $1$.  
	In this situation, the labeling must be a linear extension of the poset.  
	Therefore, it is clear that the conjecture holds for the leading term. Given a skew shape poset of size $n$ associated to the skew shape $\lambda /\mu $, the number of labeling which gives a block partition into a unique block is $\binom{n}{\lambda_1-\mu_{\ell} + \ell -1}\times(\ell -1)! \times (\lambda_1-\mu_{\ell}-1)!$, (First lets choose the set of labels of red, blue and green elements, then choose the number of way of ordering the blue ones, and then the number of wat of ordering the green ones). The latter quantity agrees with the formulas from Theorem \ref{thm:formula_c1_FMP}.
\end{remark}

\section{A probabilistc application of the statistic}

\subsection{Probabilistic approach}

Let $P$ be a fence poset. We consider the random variable $\sigma$ which follow a uniform distribution of the set of labeling of $P$.  We define the random variable $X_P := bl_P(\sigma)$. Then from  Theorem \ref{thm:Omega_to_bl}, for  $0 \leq k \leq n$, it follows that 
\begin{equation}
\label{eq:prob interpretation c_k}
\mathbb{P}(X_P = k) = c_k(P).
\end{equation}

This give a formula for the coefficients of the Ehrhart polynomial of the order polytope. 

\begin{lemma}
    Let $\mathcal{E}(P; t) = \sum_{k=0}^{n} i_k(P)t^{k}$ be the Ehrhart polynomial of the order polytope $\mathcal{O}(P)$. Then for all $0 \leq k \leq n$,
    \[
    i_k(P) =\mathbb{E}\left[ \binom{X_p}{k} \right].
    \]
\end{lemma}

\begin{proof}
    Recall that $\mathcal{E}(P; t) = \Omega(P;t+1)$. 
    Hence, 
    \begin{align*}
        i_k(P) &= [t^k]\,\Omega(P,t+1) 
        =\sum_{j=0}^{n} c_j(P)\binom{j}{k}. 
    \end{align*}
Using the probabilistic interpretation in \eqref{eq:prob interpretation c_k} of $c_j(P)$ and the definition of $\mathbb{E}[\cdot]$ we then have 
\begin{align*}
       i_k(P) = \sum_{j=0}^{n} \mathbb{P}(X_P = j)\binom{j}{k} 
        =\mathbb{E}\left[ \binom{X_p}{k} \right].
        \end{align*}
\end{proof}

As a direct application, we have $i_1(P) = E[X_P]$ and we can apply the linearity of expectation. For $x \in P$, let $Y_x$ be the following random variable 
\[
Y_x := \begin{cases}
    1 \quad \quad \textnormal{if } x \textnormal{ is the root of its block},  \\ 
    0  \quad \quad \textnormal{otherwise}. 
\end{cases}
\]

We have $X_P = \sum_{x \in P} Y_x$, therefore,

\begin{proposition}
    Let $P$ be a fence poset of size $n$, then
    \[ 
    i_1(P) = \sum_{x \in P} \mathbb{E}[Y_x] = \sum_{x \in P}\mathbb{P}(x \textnormal{ is the root of its block}).\] 
\end{proposition}

From this statement we can derive the following inequality. Let $H(n) := \sum_{k=1}^{n} \frac{1}{k+1}$.

\begin{theorem}
 Let $P$ be a fence poset of the form $z_0 < z_1<...<z_{u_1} > z_{u_1+1}>... >z_{u_1+u_2}<.... $. 
    With $n = u_1 + ...+u_r$, then  
    \[i_1(P) \geq 1+\sum_{k=1}^{r}{H({u_k}}),\]
    with equality if and only if $P$ is a chain.
\end{theorem}

\begin{proof}
    Notice that if $x$ is an ascent (resp. descent), the label of the root of the block containing $x$ is smaller or equal (resp. greater or equal) than the label of $x$. Now, take a sequence, $z_{u_k-1} > z_{u_k} < z_{u_k+1}< \dots < z_{u_{k}+u_{k+1}}$. For $1 \leq j \leq u_{i+1}$, if for all $0 \leq j' < j$,  we have $\sigma(z_{u_k+j'}) < \sigma(z_{u_k+j})$, then the label of the root of the block containing $z_{u_k + j -1 }$  is smaller than $\sigma(z_{u_k+j})$, hence, $z_{u_k+j}$ is the root of its block. Thus the probability of the event $\sigma(z_{u_k+j'}) < \sigma(z_{u_k+j})$  for all $0\leq j'<j$ is the following:
    \[
\mathbb{P}\big(\sigma(z_{u_k+j'}) < \sigma(z_{u_k+j}\big)  \text{ for all } 0\leq j'<j) = \frac{1}{j+1}.
    \]
    
    The same reasoning apply for descent chains. Therefore, for all $0 \leq k < r$, with the convention that $u_0 = 0$, 
    \[
    \mathbb{P}(z_{u_k +j} \textnormal{ is the root of its block}) \geq \frac{1}{j+1}, 
    \]  
    This inequality is sharp. This is an equality if and only if, $z_{u_{k}+j}$ is in the first chain (i.e, $k=0$).
    Therefore
\begin{align*}
    i_1(P) &= \sum_{k=0}^{r-1} \sum_{j = 1}^{u_k} \mathbb{P}(z_{u_k+j} \textnormal{ is the root of its block}) \\
    &\geq \sum_{k=0}^{r-1} \sum_{j = 1}^{u_k} \frac{1}{j+1} 
    = \sum_{k=0}^{r-1} H(u_{k+1}).
\end{align*}
With equality if and only if $P$ is a chain. 
    \end{proof}

\section{Lattice path matroids}
\label{sec:LPM}

Lattice path matroids form a rich family of matroids first studied in \cite{BonindeMierNoy} that includes several important classes, notably Schubert matroids and uniform matroids. Their associated base polytopes form an important and actively studied class of polytopes, including hypersimplices as a fundamental example. 

Recently, Ferroni, Morales and Panova proved the Ehrhart positivity of the base polytope of lattice path matroids by expressing this as a sum of order polynomial of some fence posets \cite{ferroni2026ehrhartpositivitylatticepath}. First, we proved that this induce a formula of the Ehrhart polynomial of this polytope as a sum over all the vertices of the polytope of coefficient-wise nonnegative polynomial. We then use Theorem \ref{thm:Omega_to_bl}  to give a combinatorial interpretation of the coefficients of the Ehrhart polynomial of these polytopes. In particular, this answers an open question of Stanley in \cite[Ex. 4.62(h)]{EC1} for the hypersimplex. 

\begin{definition}
Let $\lambda/\mu$ be a skew shape, and let $L$ and $U$ be its southeast and northwest boundary lattice paths, respectively, viewed as non-crossing paths from $(0,0)$ to $(m,r)$. The \new{lattice path matroid associated with $\lambda/\mu$}, denoted $M_{\lambda/\mu}$, is the lattice path matroid $M[L,U]$ whose bases are the sets of positions of the north steps of lattice paths from $(0,0)$ to $(m,r)$ that lie weakly between $L$ and $U$. The \new{length} of $M[L,U]$ is the the number of north or east steps of any path of the base of the matroid. Schubert matroids consist of lattice path matroid where $U$ is $k$ North steps followed by $n-k$ East steps. 
\end{definition}

\begin{figure}[ht]
    \centering
    \includegraphics[width=0.2\linewidth]{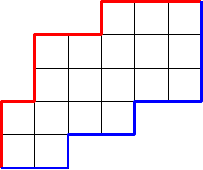}
    \caption{The upper path $U = NNENNEENEEE$ and the lower path $L = EENEENEENNN$ for the skew shape $ \lambda /\mu = 66642/311$ of length $11$. A lattice path is in $M[L,U]$ if it stays between the blue and the red path. }
    \label{fig:lattice path}
\end{figure}
Given a lattice path matroid $M[L,U]$ of length $n$, let $\mathcal{B}(M)$ denote its set of bases, identified with the lattice paths between $L$ and $U$. For a basis $x \in \mathcal{B}(M)$, let $v(x) \in \{0,1\}^n$ denote its indicator vector, where
\[
v_i(x) =
\begin{cases}
1, & \text{if the $i$-th step of $x$ is a north step},\\
0, & \text{otherwise}.
\end{cases}
\]
The base polytope of $M$ is the convex hull of the indicator vectors of its bases:
\[
\mathcal{P}(M)
:=
\operatorname{conv}
\bigl\{
v(x) \mid x \in \mathcal{B}(M)
\bigr\}.
\] 
When $U$ is $k$ North steps followed by $n-k$ East steps while $L$ is $n-k$ East steps followed by $k$ North steps, the base polytope is known as the hypersimplex. Stanley \cite[Problem $4.6.2$ (h)]{EC2} conjectured that the Ehrhart polynomial of the hypersimplex have non-negative coefficients and be proven in \cite{ferroni2026ehrhartpositivitylatticepath}. In this article, we will give a combinatorial interpretation of those coefficients. We also denote $\textcolor{blue}{\mathbf{Ehr}(M;t)}$ its Ehrhart polynomial.

A \new{ribbon} is a skew shape that is connected by edge adjacency and contains no $2\times2$ block of cells. Given a lattice path matroid $M[L,U]$ associated to a skew shape $\lambda / \mu$, we denote by $\Gamma(M[L,U])$ the set of ribbons included in $\lambda / \mu$ that contain the southwest cell and northeast cell of $\lambda / \mu$. 
Following the notation of \cite{ferroni2026ehrhartpositivitylatticepath}, given a skew shape $\lambda / \mu$, the maximum ribbon $\gamma_{\max} \in \Gamma(M[L,U])$ is obtained by starting at the southwest cell of $\lambda /\mu$ and moving north whenever possible, moving east only when an north step would leave the skew shape.

\begin{figure}[ht]
    \centering
    \begin{minipage}{0.46\textwidth}
        \centering
        \includegraphics[width=0.2\linewidth]{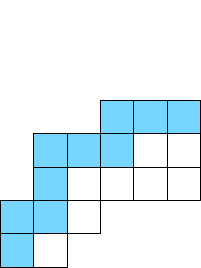}
        \caption{Minimal ribbon.}
        \label{fig:minimal_ribbon}
    \end{minipage}
    \hfill
    \begin{minipage}{0.53\textwidth}
        \centering
        \includegraphics[width=0.2\linewidth]{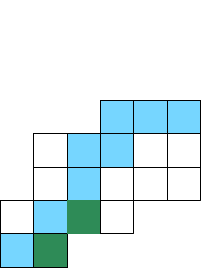}
        \caption{A ribbon and its high peaks.}
        \label{fig:high_peaks}
    \end{minipage}
    
\end{figure}

Given a ribbon $\gamma \in \Gamma(M[L,U])$, a \new{peak} of $\gamma$ is a pair of an east step directly followed by an north step. 
The \new{high peaks} of $\gamma$ are the peaks of $\gamma$ that are not also peaks of $\gamma_{\max}$.
We denote by $\textcolor{blue}{\mathsf{hp}(\gamma)}$, the set of high peaks of $\gamma$.
We also identify every cells of a given skew shape $\lambda' / \mu' $ with the associated element in the skew shape poset $P_{\lambda' / \mu'}$. It means that a cell $c_1$ cover a cell $c_2$ if and only if $c_2$ is directly at the West or at the North of $c_1$. We denote by $\textcolor{blue}{\mathcal{F}(M[L,U] )}$ (resp. $\textcolor{blue}{\mathcal{I}(M[L,U])}$)  the set of filters (resp. ideals) of the the skew shape poset associated to $\lambda / \mu$. A lattice path $x \in \mathcal{B}(M)$ separates the latter poset into an ideal $I(x)  \in \mathcal{I}(M[L,U])$ and a filter. The function $I$ is a bijection from $\mathcal{B}(M)$ to $\mathcal{I}(M[L,U])$. 
By noticing that for any poset $P$, $\Omega(P;t+1) = \sum_{I \text{ is an ideal of } P}\Omega(I; t)$, we can derive the following expression for $\mathbf{Ehr}{(M[L,U];t)}$. 

\begin{theorem}[{\cite{ferroni2026ehrhartpositivitylatticepath} Thm. 3.12}]
\label{thm:lattice_path_fmp}
Let $M[L,U]$ a lattice path matroid. Then, it's Ehrhart polynomial $\mathbf{Ehr}(M[L,U])$ can be expressed as a sum of order polynomials of fence posets. 
    \[
\mathbf{Ehr}{(M[L,U];t)} = \sum_{\substack{\gamma\in\Gamma(M[L,U])}} \sum_{\substack{F\subset\gamma\setminus\mathsf{hp}(\gamma) \\ F \textnormal{ is a filter }}}
\Omega(\gamma\setminus F; t).
\]

\begin{figure}
    \centering
    \includegraphics[width=0.40\linewidth]{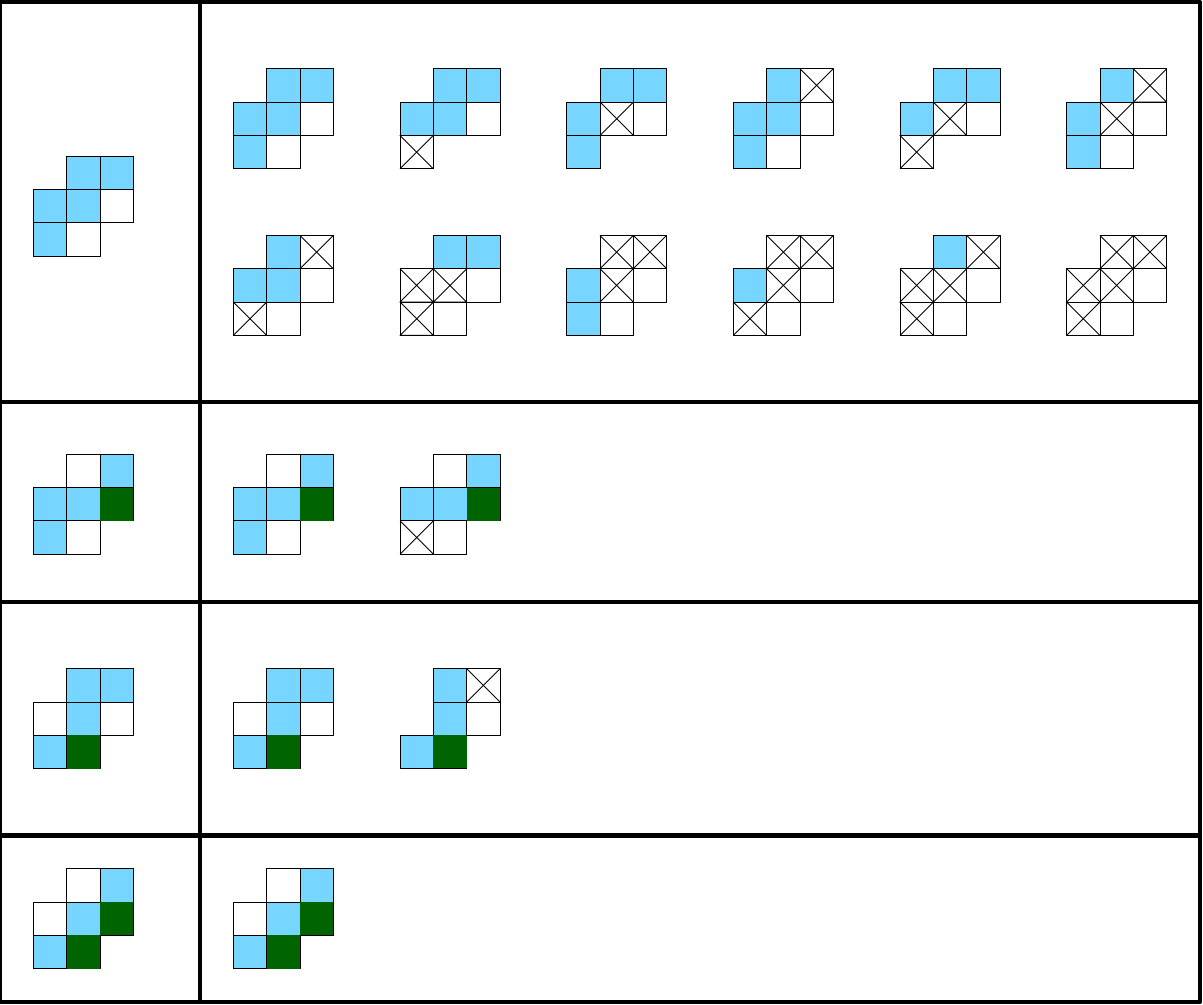}
    \caption{The four ribbons for the skew shape $332/1$ and the $17$ fence posets occurring in the sum of Theorem \ref{thm:lattice_path_fmp}.}
    \label{fig:placeholder1}
\end{figure}

Let $M[L,U]$ be a lattice path matroid associated to a skew shape $\lambda/\mu $ and let 
\[
\mathcal{S}(M[L,U]) := \bigcup_{\substack{\gamma\in\Gamma(M[L,U])}} \bigcup_{\substack{F\subset\gamma\setminus\ \textbf{hp}(\gamma) \\ F \textnormal{ is a filter }}}
\{\gamma\setminus F\}. 
\]

First, lets compute the cardinality of $\mathcal{S}(M[L,U])$.

\begin{lemma}
\label{lem:card_S_I}
    Let $\lambda /\mu $ a skew shape, then $|\mathcal{S}(M[L,U])| = |\mathcal{I}(M[L,U])|$. 
\end{lemma}
\begin{proof}
Take a lattice path matroid $M[L,U]$ with associated skew shape $\lambda /\mu$. Since the polytope of a lattice path matroid is a $0,1$-polytope, there is no integer points in the interior of this polytope, therefore : 
\[
        |\mathcal{I}(\lambda/\mu)| = |\mathcal{B}(M[L,U])|= |\textnormal{vertices}(\mathcal{P}(M[L,U]))|   =\mathbf{Ehr}\bigl(\mathcal{P}(M[L,U]);1) 
\]
By applying Theorem \ref{thm:lattice_path_fmp}, the latter expression can be written as : 
\[
        \bigcup_{\substack{\gamma\in\Gamma(M[L,U])}} \bigcup_{\substack{F\subset\gamma\setminus\ \textbf{hp}(\gamma) \\ F \textnormal{ is a filter }}} \Omega(\gamma \setminus F, 1) = |\mathcal{S}(M[L,U])|.
\]
\end{proof}
\end{theorem}
 Let $\gamma \in\Gamma(M[L,U])$ and let $F$ be a filter of $\gamma \setminus \langle \mathbf{hp}(\gamma) \rangle$, define $ \textcolor{blue}{ \phi }(\gamma \setminus F)$ as the smallest ideal of $\lambda/ \mu$ containing $(\gamma \setminus F)$. Let $I \in I(\lambda / \mu)$ an ideal. Let \textcolor{blue}{$\mathsf{sh}(I)$} be the set of cells  $c \in \lambda / \mu $ such that the cell at the northwest of $c$ is not in $I$. We also view $\mathsf{sh}(x)$ as a subposet of the skew shape poset $P_{\lambda/ \mu}$, notice that it is a fence poset. Figure \ref{fig:lattice path matroid} gives an example of $\phi$ and $\mathsf{sh}$.

\begin{figure}[ht]
    \centering
    \includegraphics[width=0.3\linewidth]{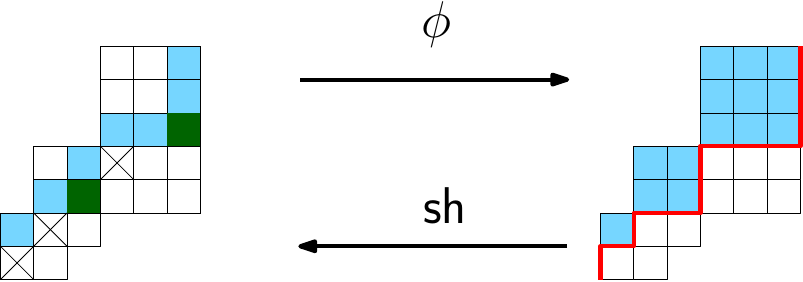}
    \caption{Example of the bijections $\phi$ and $\mathsf{sh}$.}
    \label{fig:lattice path matroid}
\end{figure}

\begin{lemma}
\label{lem:bij_S}
Let $M[L,U]$ be a lattice path matroid associated to the the skew shape $\lambda /\mu $, then $\mathsf{sh}$ is a bijection from 
     $\mathcal{I}(M[L,U])$ to $S(M[L,U])$ with $\mathsf{sh}^{-1} = \phi$. 
\end{lemma}

\begin{proof}
Let $I \in \mathcal{I}(M[L,U])$.By definition, $\mathsf{sh}(I)$ does not contain a $(2 \times 2)$-square. Hence, $\mathsf{sh}(I)$ is in $S(M([L,U]))$.

Since the two sets have the same cardinality (Lemma \ref{lem:card_S_I}), we only have to prove that $\phi( \mathsf{sh}(I)) =(I)$. Notice that $ \textnormal{max}(\mathsf{sh}(I)) = \textnormal{max}(I)$. Hence $\phi(\mathsf{sh}(I)) = I$. 
\end{proof}

Combining Theorem \ref{thm:lattice_path_fmp} and Lemma \ref{lem:bij_S}, we can therefore express the Ehrhart polynomial of a a lattice path matroid using as a sum over its elements.

\ehrhartLPM

Since $\mathsf{sh}(x)$ is a fence poset, Theorem \ref{thm:Omega_to_bl} gives a combinatorial interpretation of the coefficients of $\mathbf{Ehr}{(M[L,U]; t)}$. 

\begin{corollary}
     Let $M[L,U]$ be a lattice path matroid associated to the the skew shape $\lambda /\mu $, then 
     \[    \mathbf{Ehr}{(M[L,U]; t)} =  \sum_{x \in \mathcal{B}(M[L,U])} \sum_{\sigma \in S_{|\mathsf{sh}(x)|}} \frac{t^{bl_{\mathsf{sh}(x)}(\sigma)}}{|\mathsf{sh}(x)|!},
     \]
where $|\mathsf{sh}(x)|$ denotes the number of cells of $\mathsf{sh}(x)$. 
\end{corollary}

\begin{figure}[ht]
    \centering
    \includegraphics[width=0.6\linewidth]{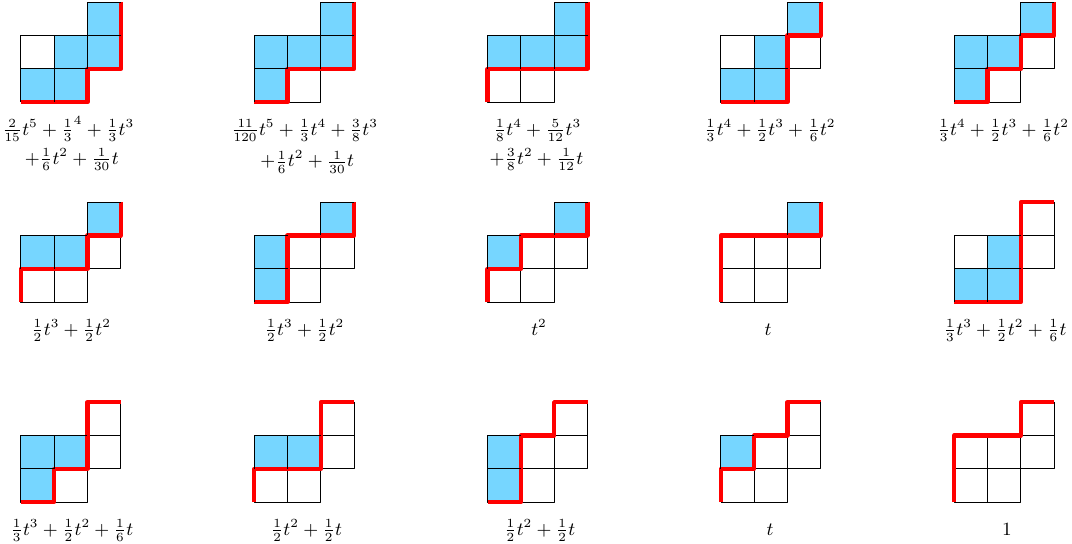}
    \caption{Decomposition of $\mathbf{Ehr}(M[L,U])$ as a sum of order polynomials for the lattice path Matroid $M[L,U]$ with $L = EENENN$ and $U = NNEENE$.}
    \label{fig:placeholder2}
\end{figure}
\subsection{Schubert matroids and hypersimplices}

Schubert matroid can been seen as lattice path matroid where the upper path $U$ consists of $n-k$ north steps directly followed by $k$ east steps. Alternatively, it can be described as lattice path matroid associated to the straight shape $\lambda$. Notice that if $M$ is a Schubert matroids, and for $x \in \mathcal{B}(M)$, $\mathsf{sh}(x)$ is connected. Using the fact that a lattice path matroid has the same Ehrhart polynomial as its dual, we can create a new statistic that gives an combinatorial interpretation of the coefficients of the Ehrhart polynomial of the base polytope of Schubert matroids. Let $M$ be the dual of a Schubert matroid. Given a lattice path $x \in \mathcal{B}(M)$, let \textcolor{blue}{$P_x$} be the fence poset obtained by taking the lattice points on the path of $x$. 

Therefore we define the following statistic on fence posets. 

\begin{definition}
     The \new{truncated poset} $\mathsf{Tr}(P)$ of a fence poset $P$ is the fence poset obtained after removing the leftmost element of $P$, all descent elements on the leftmost branch of $P$ and all ascent elements on the rightmost branch of $P$. Let $\sigma'$ be a labeling of all elements of $P$ except the leftmost one. The \new{truncated block statistic} $\mathsf{Tr}Bl_P(\sigma')$ is the block statistic of the truncated poset (see Figure \ref{fig:truncated_posets_comparaison}). 
     \[ \mathsf{Tr}Bl_P(\sigma') := Bl_{\mathsf{Tr}(P)}(\sigma'_{|{\mathsf{Tr}(P)|}}).\]
\end{definition}

\begin{figure}[ht]
    \centering
    \begin{minipage}{0.6\textwidth}
    \begin{subfigure}{0.43\textwidth}
        \centering
        \includegraphics[width=\textwidth]{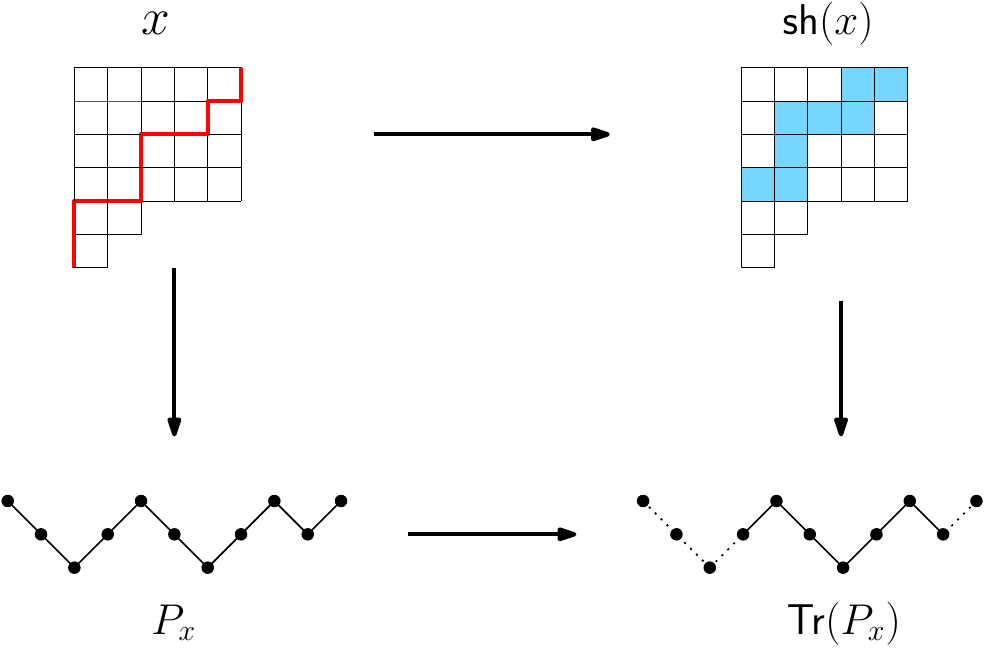}
        \label{fig:truncated_posets}
    \end{subfigure}
    \hfill
    \raisebox{0.8cm}{%
    \begin{subfigure}{0.43\textwidth}
        \centering
        \includegraphics[width=\textwidth]{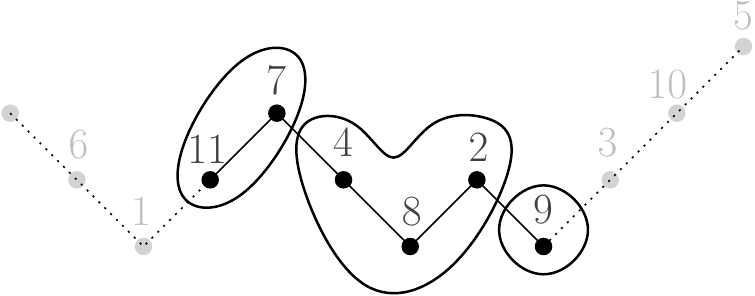}
        \label{fig:truncated_poset}
    \end{subfigure}%
    }
    \end{minipage}
    \caption{On the left, the figure explains that the truncated poset $\mathsf{Tr}(P_x)$ is the fence poset of $\mathsf{sh}(x)$. On the right, an example of the truncated block statistic.}
    \label{fig:truncated_posets_comparaison}
\end{figure}

This give a combinatorial interpretation of the coefficients of the Ehrhart polynomial of Schubert matroids. 

\schubertLPM

\begin{proof}
    Applying Proposition \ref{coro:ehrLPM}, we have 
    \[\mathbf{Ehr}(M[U];t) = \sum_{x \in M[U]}\Omega(\mathsf{sh}(x);t).\]
    Notice that $\mathsf{sh}(x)$ is the fence poset obtained from $P_x$ by removing the leftmost element of $P_x$, all descent elements on the leftmost branch of $P_x$ and all ascent elements on the rightmost branch of $P_x$. Therefore, using Theorem \ref{thm:Omega_to_bl}, \[ \Omega(\mathsf{sh}(x);t) = \frac{1}{|\mathsf{sh}(x)|!}\sum_{\sigma' \in S_{|\mathsf{sh}(x)|}}t^{bl_{\mathsf{sh}(x)}(\sigma')}.\]
    Given a labeling of $P_x$, the relative labeling of the elements of the truncated poset don't depend of the labeling of the rest of the poset, therefore 
    \[\Omega(\mathsf{sh}(x);t) = \frac{1}{|\mathsf{sh}(x)|!}\sum_{\sigma' \in S_{|\mathsf{sh}(x)|}}t^{bl_{\mathsf{sh}(x)}(\sigma')} = \frac{1}{n!}\sum_{\sigma' \in S_{n}}  t^{TrBl_{P_x}(\sigma')}.\]
\end{proof}

If $U$ consists of $k$ north steps followed by $n-k$ east steps, i.e $\mu = 0$, we recover the uniform matroid $U_{k,n}$. Its base polytope is called the hypersimplex $\Delta_{k,n}$. 

\[\Delta_{k,n} = \{x \in [0, 1]^{n} | x_1 + x_2 + \dots + x_n = k\}.\]

In this case, the set of elements $U_{k,n}$ can be as the set of fence posets with exactly $k$ descents. Therefore, we obtain as a corollary of Theorem \ref{thm:schubertLPM} a combinatorial interpretation of the coefficients of the Ehrhart polynomial of the hypersimplex $\Delta_{k,n}$ answering an open question of Stanley \cite[Problem $4.6.2$ (h)]{EC2}.

\hypersimplexLPM

For example, if $k=2$ and $n=4$, there is $6$ posets of size $n+1 = 5$ with $2$ descents. Figure \ref{fig:Delta_2_4_ex} shows all the $4!$ truncated block statistic for one those posets.  
\begin{figure}
    \centering
    \includegraphics[width=0.7\linewidth]{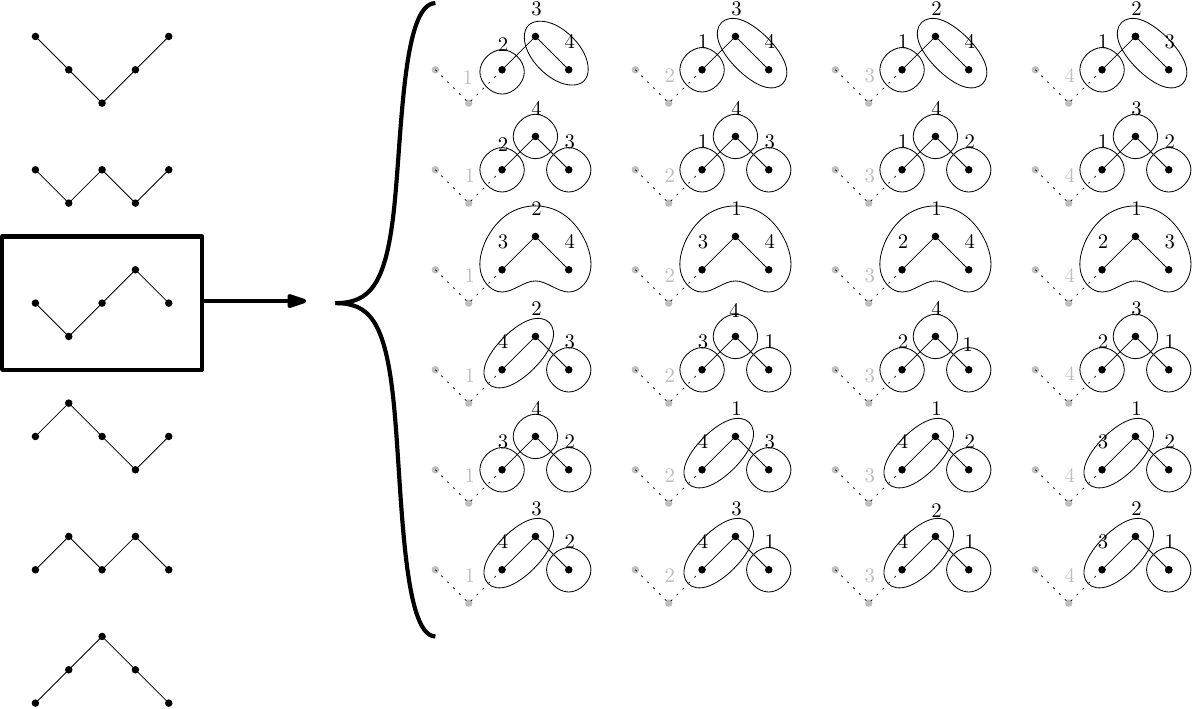}
    \caption{The $6$ fence posets of size $5$ with exactly $2$ descents, and all the 24! possible labelings for the third fence poset. }
    \label{fig:Delta_2_4_ex}
\end{figure}

\section{Final remarks}

The block statistic for fences and skew shapes was inspired by the block decomposition of Shareshian, Wright and Zhao in \cite[Proposition 2.10]{shareshian2003newapproachorderpolynomials}. However, the proof of Theorem \ref{thm:Omega_to_bl} is quite technical and is not clear how to extend it to skew shape posets in order to settle Conjecture~\ref{conj:skew shape stat}.

In contrast, Huang used tools from probability to show that the block statistic in Section~\ref{sec:variation} for circular fence posets gives the order polynomial for circular fence posets, thus proving  Conjecture~\ref{conj:circular_poset_stat}. It would also be interesting to see if this other approach could be extended to skew shapes.


\subsection{Unimodality of coefficients of $\Omega(P;t)$}
Given a fence poset $P$ of size $n$, the sequence $(a_k(P)_{0 \leq k \leq n})$ of coefficients of the order polynomial $\Omega(P;t)$ is not log-concave. The first failure of log-concavity is for a poset of size $n = 14$, with descents at position $5,7$ and $8$. We have not found yet an instance where  $(a_k(P)_{0 \leq k \leq n})$ is not unimodal but we have no reason to believe this is the case.

\subsection{A symmetric function generalization}
Given a labeled fence poset $(P, \sigma)$, we denote by $\psi(P, \sigma)$ the decreasing sequence of the block lengths of the unique block partition of $(P, \sigma)$.  
This allows us to introduce the following symmetric function $\Omega_{\mathsf{sym}}(P)$.

\[
\Omega_{\mathsf{sym}}(P) := \frac{1}{n!}\sum_{\sigma \in S_n} p_{\psi(P, \sigma)},
\]
where $p_{\lambda}$ is the power symmetric function associated to the integer partition $\lambda$. Indeed, when $P={\bf n}$ is the chain with $n$-elements then  $\Omega_{\mathsf{sym}}(\mathbf{n}) = h_n$, the complete symmetric function of degree $n$.

By taking the specialization $(x_1, x_2, \dots) =  (t,0, 0, \dots) $, we obtain 
\[ 
\Omega_{\mathsf{sym}}(P)(t,0,0,   \dots) = \frac{1}{n!}\sum_{\sigma \in S_n} t^{bl_P(\sigma)} = \Omega(P; t).
\]
It would be of interest to study this symmetric function further.



\subsection{Matroid polytopes}
Given a loopless matroid $M$ and a valuative invariant $f$, $f(M)$ can be expressed as a weighted sum of the value of $f$ on Schubert matroids, (see for example \cite[Theorem 5.9]{Ferroni_2024}). Using the truncated block statistic, this give a way to describe the coefficients of the Ehrhart polynomial 
of every loopless matroid polytope as a weighted sum of truncated block statistics $\mathsf{Tr}Bl_P(\cdot)$ on permutations. Theorem \ref{thm:schubertLPM} might be a useful tool to prove or disprove positivity conjecture of the Ehrhart polynomials of various family of matroid polytopes.

\section*{Acknowledgments} The author thanks Alejandro Morales for pointing out the open problems in \cite{ferroni2025skew,ferroni2026ehrhartpositivitylatticepath}  and for helpful comments and suggestions. The author also thank Luis Ferroni,Marni Mishna, and Félix Reutenauer for their helpful comments. The author was partially supported by the NSERC Discovery grant RGPIN-2024-06246.

\appendix

\section{Proof of Theorem \ref{thm:Omega_to_bl}}
\label{section:proof_fence}
The plan of the proof of Theorem \ref{thm:Omega_to_bl}  is the following: we start by stating two useful lemmas that will help us in the proof. Then we prove that $\dfrac{d \;A(P;t)}{dt} = \dfrac{d\;\Omega(P;t)}{dt}$ by proving by induction on $n$, the size of order polynomial and combining Lemma \ref{lem:rec_a} and Theorem \ref{thm:rec_omega}. 
\begin{lemma}
\label{lem:c_1_to_desc}
    Let $P$ a fence poset of size $n$. Let $k \in \textnormal{asc}(P), \ell \in \textnormal{desc}(P)$ such that $0 \leq k < \ell < n$. Then 

    \begin{equation}  
    (\ell-k)\cdot c_1(P_{[k, \ell-1]})=\sum_{\substack{ i \in \textnormal{desc}(P_{[k, \ell-1]})\\ k \leq i < \ell}}  c_{1}((P_{[k,\ell -1]})_{/i}) +  \mathbf{1}_{\ell = f(k)},
    \end{equation} 
where $\mathbf{1}_{a =b} := \begin{cases}
    1 \quad \quad \textnormal{if } a=b \\ 
    0 \quad \quad \textnormal{otherwise}.
\end{cases}$
\end{lemma}
\begin{proof}
    Recall that from Lemma \ref{lem:c1_fence}, given a poset $P$, with $r$ ascent element (the leftmost element of $P$ is an ascent element) and $s$ descent elements satisfies
    \[c_1(P) = \frac{(r-1)! \; (s)!}{|P|!}.\]
    If $|\textnormal{desc}(P_{[k,\ell-1]})| \geq 1$, since $ (P_{[k,\ell -1]})_{/i} = |\textnormal{desc}(P_{[k,\ell-1]})|-1$
    \begin{align*}
         \sum_{\substack{ i \in \textnormal{desc}(P)\\ k \leq i < \ell}} c_{1}((P_{[k,\ell -1]})_{/i}) &= \sum_{i \in \textnormal{desc}(P_{[k,\ell-1]})}\frac{(|\textnormal{asc}(P_{[k, \ell-1]}|-1)!\cdot(|\textnormal{desc}(P_{[k, \ell-1]})|-1)! }{(|P_{[k,\ell-1]}| -1)!}\\
         &= |P_{[k, \ell-1]}|\frac{(|\textnormal{asc}(P_{[k, \ell-1]})| -1!)\cdot |\textnormal{desc}(P_{[k, \ell-1]})|!}{|P_{[k, \ell-1]}| !} = (\ell-k)\cdot c_1(P_{[k, \ell]}).
    \end{align*}
    If $|\textnormal{desc}(P_{[k,\ell-1]})| = 0$, the sum is empty. \\
    By noticing that $|\textnormal{desc}(P_{[k,\ell-1]})| = 0$ if and  only if $\ell = f(k)$, we have 
    \[
    \sum_{\substack{ i \in \textnormal{desc}(P)\\ k \leq i < \ell}} a_{1}((P_{[k,\ell -1]})_{/i}) = \begin{cases}
        0 \quad \quad\quad \quad \quad \quad \quad \quad \, \textnormal{ if } \ell=f(k) \\
        (\ell - k )\cdot a_1(P_{[k,\ell-1]}) \quad \textnormal{ otherwise}
    \end{cases}.
    \]
    If $\ell = f(k)$, $c_1(P_{[k, \ell -1]}) = \frac{(\text{asc}(P_{[k, \ell]})-1)!}{|P_{k, \ell}|} = \frac{1}{(\ell - k)}$, and that in this case  $(\ell-k)\cdot c_1(P_{[k, \ell-1]}) = 1$. Hence the lemma holds. 
\end{proof}

\begin{lemma}
\label{lem:magic}
    Let $P$ be a fence poset of size $n$. 
    Let $0 \leq k < \ell \leq n$ $k \in \textnormal{asc}(P)$ and $\ell \in \textnormal{desc}(P) \cup \{ n\}$ and such that $f(k)  < \ell$, then 

    \begin{equation}
    \label{eq:lem_magic}
        \sum_{\substack{i \in \textnormal{desc}(P) \\ f(i)=\ell}}c_1(P_{[k,i-1]}) + \sum_{\substack{i \in \textnormal{asc}(P), k\leq i \\ f(i) = f(k)}} c_1(P_{[i, \ell-1]}) = c_1(P_{[f(k), \ell-1]}) + \sum_{\substack{i \in \textnormal{asc}(P) \\ f(i) = \ell}} c_1(P_{[k, i-1]}).
    \end{equation}
\end{lemma} 

\begin{proof}
    Since $f(k)<\ell$, we can define $x$ as the number of ascent elements in the first ascending branch and $y$ as the length of the last ascending branch of $P_{[k, \ell-1]}$. $x$ and $y$ can be equal to zero. Also let $r$ be the number of ascent elements of $P_{[k, \ell-1]}$ that are neither in the first nor the last ascending branch. Let $s$ be the number of descent elements  in $P_{[k,\ell -1]}$. For example, in Figure~\ref{fig:lemme_magique_1}, $x= 4, y=4, r = 3, s = 6$. Remark that $x+y+ r+s = \ell -k$. \\
    Now, we can express the sums directly. Recall that a poset $P$, with $r$ ascent elements and $s$ descent elements, satisfies
    \[c_1(P) = \frac{(r-1)! \; (s)!}{|P|!}.\] 
    Since there is a unique element at position $i$, with  $i \in \textnormal{desc}(P)$ such that $f(i) = \ell$, the left-hand side of Equation \ref{eq:lem_magic} is 
    \[LHS = \frac{(x+r-1)! \ (s-1)!}{(x+r+s-1)!} + \sum_{i=1}^{x} \dfrac{(i+r+y-1)!(s)!}{(i+r+y+s)!}.\]

    And the right-hand side of Equation \ref{eq:lem_magic} is 
    \[RHS = \frac{(r+y)! (s-1)!}{(r+y+s)!} + \sum_{i=0}^{y} \frac{(x+r+i-1)! (s)!}{(x+r+i+s)!}.\]

    Now, using that $\frac{(i+r+y-1)! (s)!}{(i+r+y+s)!} = \frac{(i+r+y-1)!(s-1)!}{(i+r+y+s-1)!} - \frac{(i+r+y)!(s-1)!}{(i+r+y+s)!}$, we have: 
    \begin{align*}
        LHS &=  \frac{(x+r-1)! \ (s-1)!}{(x+r+s-1)!} + \sum_{i=1}^{x} \Big( \frac{(i+r+y-1)!(s-1)!}{(i+r+y+s-1)!} - \frac{(i+r+y)!(s-1)!}{(i+r+y+s)!}\Big) \\
        &= \frac{(x+r-1)! \ (s-1)!}{(x+r+s-1)!} + \frac{(r+y)!(s-1)!}{(r+y+s)!} - \frac{(x+r+y)!(s-1)!}{(x+r+y+s)!} .
    \end{align*}
    And 
    \begin{align*}
        RHS &= \frac{(r+y)! (s-1)!}{(r+y+s)!} + \sum_{i=0}^{y} \Big( \frac{(x+r+i-1)!(s-1)!}{(x+r+i+s-1)!} - \frac{(x+r+i)!(s-1)!}{(x+r+i+s)!}\Big) \\
        &=  \frac{(r+y)! (s-1)!}{(r+y+s)!} + \frac{(x+r-1)!(s-1)!}{(x+r+s-1)!} - \frac{(x+r+y)!(s-1)!}{(x+r+y+s)!}.
    \end{align*}

    \begin{figure}[ht]
        \centering
        \includegraphics[width=0.3\linewidth]{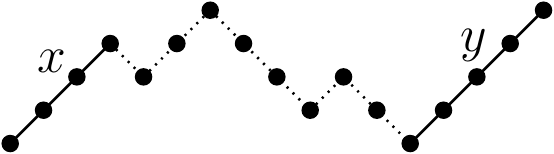}
        \caption{An example of fence poset with $x= 4, y=4, r = 3, s = 6$.}
        \label{fig:lemme_magique_1}
    \end{figure}
\end{proof}

\begin{theorem}
    for all $P$ fence posets, $\Omega(P;t) = A(P;t).$
\end{theorem}
\begin{proof}
By abuse of notation we suppress the variable $t$ so that $\Omega(P)=\Omega(P;t)$ and $A(P)=A(P;t)$.    Assume by induction that the equality holds for all fence posets of size $ \leq n-1$. 
We show that $\frac{d}{dt} \Omega(P)=\frac{d}{dt} A(P)$. In order to do this, we start by differentiating the expression for $n\cdot A(P)=n\cdot A(P_{[0,n-1]})$ in Lemma~\ref{lem:rec_a} with respect to $t$.
    \begin{equation}
    \label{eqn:main}
    \begin{aligned}
             \dfrac{d \;nA(P_{[0,n-1]})}{dt} &= \sum_{i \in \text{asc}(P)} \dfrac{dA(P_{[0,i-1]})}{dt}\cdot t \cdot A(P_{[f(i),n-1]}) + \sum_{i \in \text{asc}(P)}A(P_{[0,i-1]})\cdot t \cdot \dfrac{dA(P_{[f(i),n-1]})}{dt} \\
             &  \quad \; + \dfrac{d}{dt}\sum_{i \in \text{desc}(P)}A(P_{[0,n-1]/i}) + \sum_{i \in \text{asc(P)}} A(P_{[0,i-1]})A(P_{[f(i),n-1]}).  \\
    \end{aligned}
\end{equation}

Next, we treat each of the four summands on the RHS of \ref{eqn:main} separately and denote them by $S_1,S_2,S_3,S_4$. 
For the first summand, we use the induction hypothesis to rewrite the equation using $\Omega$ instead of $A$. We can then use Theorem \ref{thm:rec_omega} followed by some algebraic manipulation to decompose the part $\frac{d}{dt} \Omega(P_{[0,i-1]})$. We then apply the induction hypothesis and Lemma $\ref{lem:rec_a}$ followed by Theorem $\ref{thm:rec_omega}$ in order to obtain the term $\frac{d}{dt} n\Omega(P_{[0, n-1]})$ as follows.
    \begin{equation*}
        S_1 := \sum_{i \in \ \text{asc}(P)} \dfrac{d\Omega(P_{[0,i-1]})}{dt}\cdot t \cdot \Omega(P_{[f(i),n-1]}).
    \end{equation*}
        Using Theorem \ref{thm:rec_omega}, the equation becomes : 
    \begin{equation*}
        S_1 =  \sum_{i \in \ \text{asc}(P)} \hspace{-0.5cm }\sum_{\substack{0 \leq k < \ell \leq i \\ k \in \text{asc}(P), \;\ell \in \text{desc}(P) \cup \{i\}}} \hspace{-1cm} \Omega(P_{[0,k-1]})c_1(P_{[k,\ell-1]})\Omega(P_{[\ell,i-1]})\cdot t \cdot \Omega(P_{[f(i),n-1]}). 
    \end{equation*}

     We decompose the sum whether $\ell \in \textnormal{asc(P)}$ (in which case $\ell = i$ ) or $\ell \in \textnormal{desc(P)}$ (in which case $\ell < i, $). In the case $\ell=i$, notice that $\Omega(P_{[i,i-1]}) = \Omega(\emptyset) = 1.$

     \begin{align*}
        S_1 &= \sum_{\substack{0 \leq k <i \\ k \in \textnormal{asc}(P), i \in \text{asc}(P)}} \Omega(P_{[0,k-1]})c_1(P_{[k,i-1]})\cdot 1\cdot t \cdot \Omega(P_{[f(i),n-1]}) \\
        &+ \sum_{\substack{0 \leq k < \ell \leq n \\ k \in \text{asc}(P), \;\ell \in \text{desc}(P)}} \Omega(P_{[0,k-1]})\cdot c_1(P_{[k,l-1]}) \cdot \sum_{\substack{i \in \text{asc(P)}\\ \ell\ < i}}\Omega(P_{[\ell,i-1]})\cdot t\cdot \Omega(P_{[f(i),n-1]}).
\end{align*}
        
For the case $\ell < i$, since all the posets appearing in the sum have a size strictly smaller than $n$, we combine the induction hypothesis and  Lemma \ref{lem:rec_a} on $P_{[\ell,n-1]}$ .
        \begin{align*}
        S_1=& \sum_{\substack{0 \leq k <i \\ k \in \textnormal{asc}(P), i \in \text{asc}(P)}} \Omega(P_{[0,k-1]})c_1(P_{[k,i-1]})\cdot 1\cdot t \cdot \Omega(P_{[f(i),n-1]}) \\ 
        +&\sum_{\substack{0 \leq k < \ell \leq n \\k \in \text{asc}, \;\ell \in \text{desc}}} \Omega(P_{[0,k-1]})c_1(P_{[k,\ell-1]}) \Big((n-\ell)\Omega(P_{[\ell,n-1]}) - t\cdot\Omega(P_{[f(\ell), n-1]})-\sum_{i \in \text{desc}(P_{[\ell ,n-1]})} \Omega(P_{[\ell,n-1]/i})\Big).
\end{align*}
We finish treating the second summand above by splitting the term $(n-\ell)$ and applying Theorem \ref{thm:rec_omega} to recover $ \frac{d}{dt} \; n\Omega(P_{[0, n-1]})$. Thus, we obtain
\begin{align*}
        S_1=& \sum_{\substack{0 \leq k <i \\ k \in \textnormal{asc}(P), i \in \text{asc}(P)}} \Omega(P_{[0,k-1]})c_1(P_{[k,i-1]})\cdot t \cdot \Omega(P_{[f(i),n-1]}) + \frac{d \; n\Omega(P_{[0, n-1]})}{dt} \\
        +& \sum_{\substack{0 \leq k < \ell \leq n \\k \in \text{asc}, \;\ell \in \text{desc}(P)}} \Omega(P_{[0,k-1]})\cdot c_1(P_{[k,\ell-1]}) \cdot \Big((-\ell)\Omega(P_{[\ell,n-1]}) -t\cdot\Omega(P_{[f(\ell), n-1]})-\sum_{i \in \text{desc}(P_{[\ell ,n-1]})} \Omega(P_{[\ell,n-1]/i})\Big).
    \end{align*}

The strategy for the second term $S_2$ of the RHS of \ref{eqn:main} is similar to the previous case. That is, using induction hypothesis, followed by an application of Theorem \ref{thm:rec_omega} and algebraic manipulations and then recombining terms by using Lemma \ref{lem:rec_a} and again Theorem \ref{thm:rec_omega} in order to obtain terms analogous to the ones obtained in the first case.  

\[
S_2     := \sum_{i \in \text{asc}(P)}A(P_{[0,i-1]})\cdot t \cdot \dfrac{dA(P_{[f(i),n-1]})}{dt}.
\]
Apply the induction hypothesis, 
    \[
   S_2     = \sum_{i \in \text{asc}(P)}\Omega(P_{[0,i-1]})\cdot t \cdot \dfrac{d\Omega(P_{[f(i),n-1]})}{dt}. 
   \]
   Next, use Theorem \ref{thm:rec_omega} on $\dfrac{d\Omega(P_{[f(i),n-1]})}{dt}$,
\[
        S_2= \sum_{i \in \text{asc}(P)}\Omega(P_{[0,i-1]})\cdot t \cdot \sum_{\mathclap{\substack{f(i) \leq k < \ell \leq n \\ k \in \text{asc}(P)\cup\{f(i)+1\} \\ \ell \in \text{desc}(P) \cup \{n+1\}}}} \Omega(P_{[f(i), k-1]})\cdot c_1(P_{[k,\ell-1]}) \cdot  \Omega(P_{[\ell, n-1]}).
\]
decompose whether $f(i) = k$ or not,
\begin{align*}
        S_2 &=\sum_{\substack{0 \leq k < \ell \leq n\\ k \in \text{asc}(P)  \\ \ell \in \text{desc}(P) \cup \{ n+1\}}} \Omega(P_{[\ell,n-1]})\cdot c_1(P_{[k,\ell-1]}) \cdot \sum_{\substack{i \in asc(P) \\ f(i) < k}} \Omega(P_{[0,i-1]})\cdot t \cdot \Omega(P_{[f(i), k-1]}) \\
        &+ \sum_{i\in \textnormal{asc}(P)} \Omega(P_{[0,i-1]})\cdot t \cdot \sum_{\substack{f(i) < \ell \leq n \\ \ell \in \textnormal{desc}(P) \cup \{n+1\}}} 1 \cdot c_1(P_{[f(i), \ell -1]}) \cdot \Omega(P_{[\ell, n-1]}).
\end{align*}
Use Theorem \ref{lem:rec_a} on $k \cdot A(P_{[0, k-1]})$ and the induction hypothesis,
\begin{align*}
        &S_2= \sum_{\mathclap{\substack{0 \leq k < \ell \leq n\\ k \in \text{asc}(P) \\ \ell \in \text{desc}(P) \cup \{n\} }}} \Omega(P_{[\ell,n-1]})\cdot c_1(P_{[k,\ell-1]}) \cdot \Big( k\cdot \Omega(P_{[0,k-1]} - \sum_{\mathclap{\substack{i \in \text{desc}(P) \\ i \leq k}}} \Omega_{[0,k-1]/i}-\sum_{\mathclap{\substack{i \in \textnormal{asc}{(P)}\\ i\leq k, \;f(i)\geq k}}} t \cdot \Omega(P_{[0,i-1]})\Big) \\
        &+ \sum_{i\in \textnormal{asc}(P)} \Omega(P_{[0,i-1]})\cdot t \cdot \sum_{\mathclap{\substack{f(i) < \ell \leq n \\ \ell \in \textnormal{desc}(P) \cup \{n\}}}} 1 \cdot c_1(P_{[f(i), \ell -1]}) \cdot \Omega(P_{[\ell, n-1]}).
    \end{align*}

The third term $S_3$ of the RHS of \ref{eqn:main} is treated in a similar way to the two previous ones. By induction hypothesis we have 
\[
 S_3  := \dfrac{d}{dt}\sum_{i \in \text{desc}(P)}A(P_{[0,n-1]/i}) 
 \]
 Apply the induction hypothesis
 \[
 S_3 = \dfrac{d}{dt}\sum_{i \in \text{desc}(P)}\Omega(P_{[0,n-1]/i})     =\sum_{i \in \text{desc}(P)} \dfrac{d}{dt}\Omega(P_{[0,n-1]/i}). 
\]
Using Theorem \ref{thm:rec_omega}, this becomes
\[
S_3 =  \sum_{i \in \text{desc}(P)} \sum_{\substack{0 \leq k < \ell \leq n-1 \\ k \in \text{asc}(P_{/i}) \\ \ell \in \text{desc}(P_{/i}) \cup \{n-1\}}} \Omega((P_{/i})_{[0,k-1]})\cdot c_1((P_{/i})_{[k, \ell-1]})\cdot \Omega((P_{/i})_{[\ell,n-2]}). 
\]    
We decompose into three parts: $i< k$, $k \leq i \leq \ell$ and $i>\ell$. 
\begin{align*}
    S_3 &= \sum_{\substack{0 \leq k < \ell \leq n \\ k \in \text{asc}(P) \\ \ell \in \text{desc}(P) \cup \{n\}}} \sum_{\substack{i  \in \text{desc}(P) \\ i < k} }\Omega((P_{[0,k-1]})_{/i})\cdot c_1(P_{[k, l-1]})\cdot \Omega(P_{[\ell,n-1]}) \\
    &+ \sum_{\substack{0 \leq k < \ell \leq n-1 \\ k \in \text{asc}(P) \\ \ell \in \text{desc}(P)\cup \{n\}}} \sum_{\substack{i  \in \text{desc}(P) \\  \ell < i} }\Omega(P_{[0,k-1]})\cdot c_1(P_{[k, l-1]})\cdot \Omega((P_{[\ell,n-1]})_{/i}) \\
    &+ \sum_{\substack{0 \leq k < \ell \leq n-1 \\ k \in \text{asc}(P) \\ \ell \in \text{desc}(P)\cup \{n\}}} \sum_{\substack{ i \in \textnormal{desc}(P_{[k, \ell-1]}) \\ k \leq i < \ell} }\hspace{-0.5cm}\Omega(P_{[0,k-1]})\cdot c_1((P_{[k, \ell-1]})_{/i})\cdot \Omega(P_{[\ell,n-1]}).
    \end{align*}

Finally, we use Lemma $\ref{lem:magic}$ to rewrite the following term that will appear in the decomposed sum of  \eqref{eqn:main}. 
\begin{align*}
    R &:=\sum_{\mathclap{\substack{0 \leq k < \ell \leq n \\ k \in \textnormal{asc}(P) \\ \ell \in \textnormal{desc(P)} \cup \{ n\}}}} (k-l) \cdot \Omega(P_{[0, k-1]})\cdot c_1(P_{[k, \ell-1]}) \cdot \Omega(P_{[\ell, n-1]}) \\
    &= \sum_{\mathclap{\substack{0 \leq k < \ell \leq n \\ k \in \textnormal{asc}(P) \\ \ell \in \textnormal{desc(P)} \cup \{ n\}}}} \Omega(P_{[0, k-1]}) \cdot \Big(\hspace{0.1cm}\sum_{\mathclap{\substack{k \leq i < \ell \\ i \in \textnormal{desc}(P_{[k, \ell-1]})}}} c_{1}((P_{[k,\ell -1]})_{/i}) +  \mathbf{1}_{\ell = f(k)}\Big) \cdot \Omega(P_{[\ell, n-1]}). 
\end{align*}
We split the sum and notice that there exists a unique $\ell \in \textnormal{desc}(P) \cup \{ n \}$ such that $\ell = f(k)$. 
\[
    R=\hspace{-0.5cm}\sum_{\substack{0 \leq k < \ell \leq n \\ k \in \textnormal{asc}(P)\\ \ell \in \textnormal{desc(P)} \cup \{ n\}}} \sum_{\substack{ i \in \textnormal{desc}(P_{[k, \ell-1]})\\ k \leq i < \ell}} \hspace{-0.5cm}\Omega(P_{[0, k-1]})\cdot c_{1}((P_{[k,\ell -1]})_{/i})\cdot \Omega(P_{[\ell, n-1]})
    + \sum_{k \in \textnormal{asc}(P)}\Omega(P_{[0, k-1]})\cdot 1 \cdot \Omega(P_{[f(k), n-1]}).
\]
Adding the terms of the RHS of \eqref{eqn:main}, most of the terms cancel each other, we obtain: 

\begin{align*}
\dfrac{d \; nA(P_{[0,n-1]})}{dt} &= S_1 + S_2 +S_3 + \sum_{\mathclap{i \in \text{asc(P)}}} \Omega(P_{[0,i-1]})\Omega(P_{[f(i),n-1]})\\
&=\frac{d \; n\Omega(P_{[0, n-1]})}{dt} + \sum_{\mathclap{\substack{0 \leq k <i \\ k \in \textnormal{asc}(P), i \in \text{asc}(P)}}} \Omega(P_{[0,k-1]})c_1(P_{[k,i-1]})\cdot t \cdot \Omega(P_{[f(i),n-1]}) \\
    &- \sum_{\mathclap{\substack{0 \leq k < \ell \leq n \\k \in \text{asc}(P), \;\ell \in \text{desc}(P)}}} \Omega(P_{[0,k-1]})\cdot c_1(P_{[k,\ell-1]}) \cdot t\cdot\Omega(P_{[f(\ell), n-1]})\\
    &-\sum_{\mathclap{\substack{0 \leq k < \ell \leq n\\ k \in \text{asc}(P),\; \ell \in \text{desc}(P) \cup \{n\} }}} \Omega(P_{[\ell,n-1]})\cdot c_1(P_{[k,\ell-1]}) \cdot \sum_{\mathclap{\substack{i \in \textnormal{asc}{(P)}\\ i\leq k, \;f(i)\geq k}}}t \cdot \Omega(P_{[0,i-1]}) \\
    &+ \sum_{\mathclap{i\in \textnormal{asc}(P)}} \Omega(P_{[0,i-1]})\cdot t \cdot \sum_{\mathclap{\substack{f(i) < \ell \leq n \\ \ell \in \textnormal{desc}(P) \cup \{n\}}}} c_1(P_{[f(i), \ell -1]}) \cdot \Omega(P_{[\ell, n-1]}).
\end{align*}

To verify the proof, we just have to show that the sum of the four terms in the RHS above is $0$. In order to do so, we will have to make the following changes of variables. 
\begin{itemize}
    \item By setting $\ell =: f(i)$
    \begin{equation*}
        \hspace{-2cm} \sum_{\mathclap{\substack{0 \leq k <i \\ k \in \textnormal{asc}(P), i \in \text{asc}(P)}}} \Omega(P_{[0,k-1]})c_1(P_{[k,i-1]})\cdot t \cdot \Omega(P_{[f(i),n-1]}) = \sum_{\mathclap{\substack{0 \leq k < \ell \leq n \\ k \in \textnormal{asc}(P),\ell \in \textnormal{desc}(P) \cup \{ n\}}}} t \cdot \Omega(P_{[0, k-1]})\cdot \Omega(P_{[\ell, n-1]}) \sum_{\mathclap{\substack{i\in \textnormal{asc}(P)\\ f(i)=\ell}}} c_1(P_{[k, i-1]}).
    \end{equation*}
\vspace{0.5cm}
    \item By setting $\ell := f(\ell')$
    \begin{equation*}
       \hspace{-2cm} \sum_{\mathclap{\substack{0 \leq k < \ell' \leq n \\k \in \text{asc}(P) \\ \ell' \in \text{desc}(P)}}} \Omega(P_{[0,k-1]})\cdot c_1(P_{[k,\ell'-1]}) \cdot t\cdot\Omega(P_{[f(\ell'), n-1]}) = \sum_{\mathclap{\substack{0 \leq k < \ell \leq n \\ k \in \textnormal{asc}(P) \\ \ell \in \textnormal{desc}(P) \cup \{ n\}  }}} t \cdot \Omega(P_{[0, k-1]})\cdot \Omega(P_{[\ell, n-1]}) \sum_{\mathclap{\substack{i\in \textnormal{desc}(P)\\ f(i)=\ell}}} c_1(P_{[k, i-1]}).
    \end{equation*}

\bigskip

    \item By renaming $k := i'$ and $i:=k'$
    \[
       \hspace{-2cm} \sum_{\mathclap{\substack{0 \leq k' < \ell \leq n\\ k' \in \text{asc}(P) \\ \ell \in \text{desc}(P) \cup \{n\} }}} \Omega(P_{[\ell,n-1]})\cdot c_1(P_{[k',\ell-1]}) \cdot \sum_{\mathclap{\substack{i' \in \textnormal{asc}(P)\\ i'\leq k', \;f(i')\geq k'}}}t \cdot \Omega(P_{[0,i'-1]}) = \sum_{\mathclap{\substack{0 \leq k < \ell \leq n \\k \in \textnormal{asc}(P) \\ \ell \in \textnormal{desc}(P) \cup \{ n\} }}} t \cdot \Omega(P_{[0, k-1]})\cdot \Omega(P_{[\ell, n-1]}) \sum_{\mathclap{\substack{i \in \textnormal{asc}(P),\\ k\leq i,\; f(i)=f(k)}}} c_1(P_{[i,\ell-1]}).
    \]

\bigskip    
    \item By renaming $k := i$ and noticing that 
    \begin{equation*}
       \hspace{-2cm}  \sum_{\mathclap{i\in \textnormal{asc}(P)}} \Omega(P_{[0,i-1]})\cdot t \cdot \sum_{\mathclap{\substack{f(i) < \ell \leq n \\ \ell \in \textnormal{desc}(P) \cup \{n\}}}}  c_1(P_{[f(i), \ell -1]}) \cdot \Omega(P_{[\ell, n-1]}) =  \sum_{\mathclap{\substack{0 \leq k < \ell \leq n \\ k \in \textnormal{asc}(P) \\ \ell \in \textnormal{desc}(P) \cup \{ n\} }}} t \cdot \Omega(P_{[0, k-1]})\cdot \Omega(P_{[\ell, n-1]}) c_1(P_{[f(k), \ell-1]}).
    \end{equation*}
\end{itemize}
\bigskip
Using these change of variables $\dfrac{d \;nA(P_{[0, n-1]})}{dt}$ becomes the following : 

\begin{align*}
   \dfrac{d \;nA(P_{[0, n-1]})}{dt} &= \dfrac{d \;n\Omega(P_{[0, n-1]})}{dt}\\
   &+ S \cdot \Bigl(\sum_{\substack{i \in \textnormal{asc(P)} \\  f(i)=\ell}}c_1(P_{[k,i-1]})+c_1(P_{[f(k), \ell-1]})- \sum_{\substack{i\in \textnormal{desc}(P) \\  f(i)=\ell}} c_1(P_{[k, i-1]}) - \sum_{\substack{i \in \textnormal{asc}(P) \\ k \leq i, f(i)=f(k)}} c_1(P_{[i,\ell-1]}) \Bigr), 
\end{align*}
where $S := \displaystyle\sum_{\mathclap{\substack{0 \leq k < \ell \leq n \\ k \in \textnormal{asc}(P),\ell \in \textnormal{desc}(P) \cup \{ n\} }}} t \cdot \Omega(P_{[0, k-1]})\cdot \Omega(P_{[\ell, n-1]}).$
\vspace{0.5cm}

Finally, using Lemma \ref{lem:magic}, we obtain the intended equality
\[
    \dfrac{d \; nA(P_{[0,n-1]})}{dt}  = \dfrac{d \; n\Omega(P_{[0,n-1]})}{dt}, 
\]
and since $A(P;0) = 0 = \Omega(P;0)$, we conclude that $A(P;t) = \Omega(P;t)$, as desired.

\end{proof}

\printbibliography

\end{document}